\documentclass[reqno,12pt]{amsart}
\usepackage[colorlinks=true, linkcolor=blue, citecolor=blue]{hyperref}

\usepackage{amssymb}
\usepackage{amsmath, graphicx, rotating}
\usepackage{color}
\usepackage{soul}
\usepackage[dvipsnames]{xcolor}

\usepackage{ifthen}
\usepackage{xkeyval}
\usepackage{todonotes}
\usepackage[T1]{fontenc}
\usepackage{lmodern}
\usepackage[english]{babel}

\usepackage{ upgreek }
\usepackage{stmaryrd}
\SetSymbolFont{stmry}{bold}{U}{stmry}{m}{n}
\usepackage{amsthm}
\usepackage{float}

\usepackage{ bbm }
\usepackage{ stmaryrd }
\usepackage{ mathrsfs }
\usepackage{ frcursive }
\usepackage{ comment }

\usepackage{pgf, tikz}
\usetikzlibrary{shapes}
\usepackage{varioref}
\usepackage{enumitem}

\definecolor{rouge}{rgb}{0.7,0.00,0.00}
\definecolor{vert}{rgb}{0.00,0.5,0.00}
\definecolor{bleu}{rgb}{0.00,0.00,0.8}
\usepackage[margin=1in]{geometry}
\newtheorem{theorem}{Theorem}[section]
\newtheorem*{theorem*}{Theorem}
\newtheorem{lemma}[theorem]{Lemma}

\newtheorem{proposition}[theorem]{Proposition}

\labelformat{hypothesis}{\textbf{M\kern-0.1mm#1}}

\newtheorem{condition}{Condition}

\newtheorem{conditionA}{A\kern-0.1mm}
\labelformat{conditionA}{\textbf{A\kern-0.1mm#1}}

\theoremstyle{definition}
\newtheorem{example}[theorem]{Example}
\newtheorem{remark}[theorem]{Remark}

\def \eref#1{\hbox{(\ref{#1})}}

\numberwithin{equation}{section}

\def\geq{\geqslant}
\def\leq{\leqslant}

\def\RR{\mathbb{R}}
\def\PP{\mathbb{P}}
\def\EE{\mathbb{E}}

\def \eref#1{\hbox{(\ref{#1})}}

\def\EE{\mathbb{ E}}

\def\ep{{\epsilon}}

\begin{document}

\title[Strong convergence order for slow-fast MVSDE]
{Strong convergence order for slow-fast McKean-Vlasov stochastic differential equations}

\author{Michael R\"{o}ckner}
\curraddr[R\"{o}ckner, M.]{Fakult\"{a}t f\"{u}r Mathematik, Universit\"{a}t Bielefeld, D-33501 Bielefeld, Germany, and Academy of Mathematics and Systems Science,
  Chinese Academy of Sciences (CAS), Beijing, 100190, China}
\email{roeckner@math.uni-bielefeld.de}

\author{Xiaobin Sun}
\curraddr[Sun, X.]{ School of Mathematics and Statistics, Jiangsu Normal University, Xuzhou, 221116, China}
\email{xbsun@jsnu.edu.cn}

\author{Yingchao Xie}
\curraddr[Xie, Y.]{ School of Mathematics and Statistics, Jiangsu Normal University, Xuzhou, 221116, China}
\email{ycxie@jsnu.edu.cn}

\begin{abstract}
In this paper, we consider the averaging principle for a class of McKean-Vlasov stochastic differential equations with slow and fast time-scales. Under some proper assumptions on the coefficients, we first prove that the slow component strongly converges to the solution of the corresponding averaged equation with convergence order $1/3$ using the approach of time discretization. Furthermore, under stronger regularity conditions on the coefficients, we use the technique of Poisson equation to improve the order to $1/2$, which is the optimal order of strong convergence in general.
\end{abstract}

\date{\today}

\maketitle

\section{Introduction}
Let $\{W^{1}_t\}_{t\geq 0}$ and $\{W^{2}_t\}_{t\geq 0}$ be mutually independent $d_1$ and $d_2$ dimensional standard Brownian motions on a complete probability space $(\Omega, \mathscr{F}, \mathbb{P})$ and $\{\mathscr{F}_{t},t\geq0\}$ be the natural filtration generated by $W^{1}_t$ and $W^{2}_t$. Let the following maps $b=b(t,x,\mu,y)$, $\sigma=\sigma(t,x,\mu)$, $f=f(t,x,\mu,y)$ and $g=g(t,x,\mu,y)$ be given:
\begin{eqnarray*}
&&b: [0, \infty)\times\RR^n\times\mathscr{P}_2\times\RR^m \rightarrow \RR^{n};\\
&& \sigma: [0, \infty)\times\RR^n\times \mathscr{P}_2\rightarrow \RR^{n\times d_1};\\
&&f:[0, \infty)\times\RR^n\times\mathscr{P}_2\times\RR^m\rightarrow \RR^{m};\\
&&g:[0, \infty)\times\RR^n\times\mathscr{P}_2\times\RR^m\rightarrow \RR^{m\times d_2}
\end{eqnarray*}
such that $b$, $\sigma$, $f$ and $g$ are continuous in $(t,x,\mu,y)\in [0, \infty)\times\RR^n\times\mathscr{P}_2\times\RR^m$, where $\mathscr{P}_2$ is defined by
$$
\mathscr{P}_2:=\Big\{\mu\in \mathscr{P}: \mu(|\cdot|^2):=\int_{\RR^n}|x|^2\mu(dx)<\infty\Big\},
$$
where $\mathscr{P}$ is the set of all probability measure on $(\RR^n, \mathscr{B}(\RR^n))$. Then $\mathscr{P}_2$ is a polish space under the $L^2$-Wasserstein distance, i.e.,
$$
\mathbb{W}_2(\mu_1,\mu_2):=\inf_{\pi\in \mathscr{C}_{\mu_1,\mu_2}}\left[\int_{\RR^n\times \RR^n}|x-y|^2\pi(dx,dy)\right]^{1/2}, \quad \mu_1,\mu_2\in\mathscr{P}_2,
$$
where $\mathscr{C}_{\mu_1,\mu_2}$ is the set of all couplings for $\mu_1$ and $\mu_2$.

\vspace{0.1cm}
In this paper, we consider the following slow-fast McKean-Vlasov stochastic differential  equations (SDEs):
\begin{equation}\left\{\begin{array}{l}\label{Equation}
\displaystyle
d X^{\ep}_t = b(t, X^{\ep}_t, \mathscr{L}_{X^{\ep}_t}, Y^{\ep}_t)dt+\sigma(t, X^{\ep}_t, \mathscr{L}_{X^{\ep}_t})d W^{1}_t,\quad X^{\ep}_0=x\in\RR^n, \\
d Y^{\ep}_t =\frac{1}{\ep}f(t, X^{\ep}_t, \mathscr{L}_{X^{\ep}_t}, Y^{\ep}_t)dt+\frac{1}{\sqrt{\ep}}g(t, X^{\ep}_t, \mathscr{L}_{X^{\ep}_t}, Y^{\ep}_t)d W^{2}_t,\quad Y^{\ep}_0=y\in\RR^m,
\end{array}\right.
\end{equation}
where $\mathscr{L}_{X^{\ep}_t}$ is the law of $X^{\ep}_t$, $\ep$ is a small and positive parameter describing the ratio of the time scale between the slow component $X^{\ep}_t\in\RR^n$ and fast component $Y^{\ep}_t\in\RR^m$.

\vspace{0.1cm}
The averaging principle has a long and rich history in multiscale models, which have wide applications in material sciences, chemistry,
fluid dynamics, biology, ecology, climate dynamics etc., see e.g., \cite{BR,EW,ELV,HKWJ,MHR,WTRY16} and references therein. The averaging principle is essential to describe the asymptotic behavior of the slow component as $\ep\to 0$, i.e., the slow component will convergence to the so-called averaged equation. Bogoliubov and Mitropolsky \cite{BM} first studied the averaging principle for deterministic systems. The averaging principle for SDEs was first studied by  Khasminskii in \cite{K1}, see e.g., \cite{GKK,G,GL,L,LRSX,V0} for further developments. The averaging principle for slow-fast stochastic partial differential equations (SPDEs) was first investigated by Cerrai and Freidlin in \cite{CF}, see e.g., \cite{B1,C1,C2,CL,DSXZ,FLL,FWLL,GP,GP1,GP2,LRSX2,WR12,WRD12} for further developments.

\vspace{0.1cm}
The McKean-Valsov SDEs (also called distribution dependent SDEs) describe stochastic systems whose evolution is determined by both the microcosmic location and the macrocosmic distribution of the particle. The time marginal laws of the solution of such SDEs satisfies a nonlinear Fokker-Planck-Kolmogorov equation. The existence and uniqueness of weak and strong solutions have been studied intensively (see \cite{MV, WFY} and references therein). Further properties, such as the Harnack inequality or the Bismut formula for the Lions Derivative have been investigated in \cite{WFY} and \cite{RW} respectively. However, to the authors' knowledge, this paper is the first in which the averaging principle for two-time scale distribution dependent SDEs is considered. 

\vspace{0.1cm}
For numerical purposes, however, only studying the strong convergence of the slow component to the corresponding averaged equation is not enough, since in addition one needs to know the rate of convergence. Hence, the main purpose of our paper is to study the strong convergence rate for two-time scale distribution dependent SDEs. More precisely, one tries to find the largest possible $\alpha>0$ such that
\begin{eqnarray}
\sup_{t\in [0, T]}\mathbb{E}|X^{\ep}_t-\bar{X}_t|^2\leq C\ep^{\alpha},\label{main result}
\end{eqnarray}
where $C$ is a constant depending on $T, |x|, |y|$, and $\bar{X}$ is the solution of the corresponding averaged equation  (see Eq. \eref{1.3} below).

\vspace{0.1cm}
In the distribution-independent case, the strong convergence rate for two-time scale stochastic system has been studied in a number of papers (see e.g., \cite{GKK,G, L,RSX} for the finite dimensional case, and \cite{B1, B2} for the infinite dimensional case). The approach based on Khasminskii's technique of time discretization is often used to study the strong convergence rate (see \cite{B1,GKK,G, L}). Recently, the technique of Poisson equation has been used to study the strong convergence rate in \cite{B2,RSX}, and the optimal convergence order was obtained in general. Motivated by this, in this paper we will use the techniques of time discretization and Poisson equation to study the strong convergence rate for two-time scale distribution dependent SDEs separately. More precisely, under some proper assumptions on the coefficients, we use the technique of time discretization to obtain the convergence order $1/3$, which is however usually not the optimal order. It turns out that under some stronger assumptions  on the coefficients, the optimal convergence order $1/2$ can indeed be obtained by the method of Poisson equation.

\vspace{0.1cm}
If applying the technique of Poisson equation (see \cite{PV1,PV2,RSX}) to prove our main result, the main difficulty is to analyse the regularity of the solution $\Phi(t,x,\mu,y)$ of the corresponding Poisson equation with respect to ($w.r.t.$) the parameter $\mu$. Indeed, this method highly depends on the regularity of $\Phi$ $w.r.t.$ parameters. However, due to the coefficients dependence on the distribution, $\Phi$ will also depend on the distribution $\mu$. Unlike as for classical SDEs, we have to apply It\^o's formula to $\Phi$ composed with the process $(t, X^{\ep}_t,\mathscr{L}_{X^{\ep}_t}, Y^{\ep}_t)$, which in particular, means that we have to differentiate in the measure $\mu$. As a consequence, some additional terms involving the Lions derivative of $\Phi$, so we have to estimate the regularity of $\Phi$ $w.r.t.$ the parameter  $\mu$ carefully.

\vspace{0.1cm}
The paper is organized as follows. In the next section, we introduce some notation and assumptions that we use throughout the paper, and present out the main results. Sections 3 and 4 are devoted to proving the strong convergence rate by using the techniques of time discretization and Poisson equation respectively. We give an example in Section 5. In the Appendix, we give the detailed  proof of the existence and uniqueness of solutions for our system and prove some important estimates.

\vspace{1mm}
We note that throughout this paper $C$ and $C_T$  denote positive constants which may change from line to line, where the subscript $T$ is used to emphasize that the constant depends on $T$.

\section{Notations and main results}\label{sec.prelim}

Now, we first remind the reader of the definition of differentiability on the Wasserstein space. Following the idea in \cite[Section 6]{C}, for $u: \mathscr{P}_2\rightarrow \RR$ we denote by $U$ its "extension" to $L^2(\Omega, \PP;\RR^n)$ defined by
$$
U(X):=u(\mathscr{L}_{X}),\quad X\in L^2(\Omega,\PP;\RR^n).
$$
Then we say that $u$ is differentiable at $\mu\in\mathscr{P}_2$ if there exists $X\in L^2(\Omega,\PP;\RR^n)$ such that $\mathscr{L}_{X}=\mu$ and $U$ is Fr\'echet differentiable at $X$. By Riesz' theorem, the Fr\'echet derivative $DU(X)$, viewed as an element of $L^2(\Omega,\PP;\RR^n)$, can be represented as
$$
DU(X)=\partial_{\mu}u(\mathscr{L}_{X})(X),
$$
where $\partial_{\mu}u(\mathscr{L}_{X}):\RR^n\rightarrow \RR^n$, which is called Lions derivative of $u$ at $\mu= \mathscr{L}_{X}$. Moreover, $\partial_{\mu}u(\mu)\in L^2(\mu;\RR^n)$, for $\mu\in\mathscr{P}_2$. Furthermore, if $\partial_{\mu}u(\mu)(z):\RR^n\rightarrow \RR^n$ is differentiable at $z\in\RR^n$, we denote its derivative by $\partial_{z}\partial_{\mu}u(\mu)(z):\RR^n\rightarrow \RR^n\times\RR^n$.

\vspace{0.1cm}
Let $|\cdot|$ be the Euclidean vector norm, $\langle\cdot, \cdot\rangle$ be the Euclidean inner product and $\|\cdot\|$ be the matrix norm or the operator norm if there is no confusion possible. We call a vector-valued, or matrix-valued function $u(\mu)=(u_{ij}(\mu))$ differentiable at $\mu\in\mathscr{P}_2$, if its all its components are  differentiable at $\mu$, and set $\partial_{\mu}u(\mu):=(\partial_{\mu}u_{ij}(\mu))$ and $\|\partial_{\mu}u(\mu)\|^2_{L^2(\mu)}:=\sum_{i,j}\int_{\RR^n}|\partial_{\mu}u_{ij}(\mu)(z)|^2\mu(dz)$. Furthermore, we call $\partial_{\mu}u(\mu)(z)$ differentiable at $z\in\RR^n$, if all its components are differentiable at $z$, and set $\partial_{z}\partial_{\mu}u(\mu)(z):=(\partial_{\mu}u_{ij}(\mu)(z))$ and $\|\partial_{z}\partial_{\mu}u(\mu)\|^2_{L^2(\mu)}:=\sum_{i,j}\int_{\RR^n}\|\partial_{z}\partial_{\mu}u_{ij}(\mu)(z)\|^2\mu(dz)$. For convenience, we write $u\in C^{1,1}(\mathscr{P}_2, \RR^n)$, if the $\RR^n$-valued map $\mu\mapsto u(\mu)$ is differentiable at any $\mu\in\mathscr{P}_2$, and $\partial_{\mu}u(\mu)(z):\RR^n\rightarrow \RR^n$ is differentiable at any $z\in\RR^n$.

\vspace{0.1cm}
For a vector-valued or matrix-valued function $F(t,x,y)$ defined on $[0,\infty)\times\RR^n\times\RR^m$. For any $u, v\in \{t,x,y\}$, we use $\partial_u F$ to denote the first order partial derivative of $F$ $w.r.t.$ component $u$ and $\partial^2_{u v} F$ to denote its second order partial derivatives of $F$ $w.r.t.$ components $u$ and $v$. For convenience, we say an $\RR^n$-valued $F$ belongs to $C^{1,2,2}([0,\infty)\times\RR^n\times\RR^m,\RR^n)$, if $\partial_t F(t,x,y)$, $\partial^2_{xx} F(t,x,y)$ and $\partial^2_{yy} F(t,x,y)$ exist for any $(t,x,y)\in[0,\infty)\times\RR^n\times\RR^m$.

\vspace{0.3cm}
We suppose that for any $T>0$, there exist constants $C_T, \beta\in (0,\infty)$ and $\gamma_1,\gamma_2\in (0,1]$ such that the following conditions hold for all $t, t_1, t_2\in [0, T], x,x_1,x_2\in\RR^n, \mu,\mu_1,\mu_2\in \mathscr{P}_2, y,y_1,y_2\in\RR^m$.

\smallskip
\noindent
\begin{conditionA}\label{A1} (Conditions on $b$, $\sigma$, $f$ and $g$ )
\begin{eqnarray}
&&|b(t_1, x_1, \mu_1, y_1)-b(t_2, x_2, \mu_2, y_2)|+\|\sigma(t_1, x_1,\mu_1)-\sigma(t_2, x_2,\mu_2)\|\nonumber\\
\leq \!\!\!\!\!\!\!\!&&C_T\left[|t_1-t_2|+|x_1-x_2|+|y_1-y_2|+\mathbb{W}_2(\mu_1, \mu_2)\right];\label{A11}
\end{eqnarray}
\begin{eqnarray}
&&|f(t_1, x_1,\mu_1, y_1)-f(t_2, x_2, \mu_2,y_2)|+\|g(t_1, x_1, \mu_1,y_1)-g(t_2, x_2, \mu_2, y_2)\|\nonumber\\
\leq \!\!\!\!\!\!&&C_T\left[|t_1-t_2|+|x_1-x_2|+|y_1-y_2|+\mathbb{W}_2(\mu_1, \mu_2)\right];\label{A21}
\end{eqnarray}
and
\begin{eqnarray}
2\langle f(t,x,\mu, y_1)-f(t,x, \mu,y_2), y_1-y_2\rangle\!+\!3\|g(t, x, \mu,y_1)-g(t,x,\mu, y_2)\|^2\!\leq -\beta|y_1-y_2|^2.\label{sm}
\end{eqnarray}
\end{conditionA}

\smallskip
\noindent
\begin{conditionA}\label{A2}(Conditions on first-order partial derivatives)
The first-order partial derivatives $\partial_t b(t,x,\mu,y)$, $\partial_x b(t,x,\mu,y)$, $\partial_{\mu} b(t,x,\mu,y)$ and $\partial_y b(t,x,\mu,y)$ exist for any $(t,x,y,\mu)\in[0,\infty)\times\RR^n\times\RR^m\times\mathscr{P}_2$. Moreover,
\begin{eqnarray}
\sup_{t\in[0,T],x\in\RR^n,\mu\in\mathscr{P}_2}|\partial_t b(t, x, \mu, y_1)-\partial_t b(t, x, \mu, y_2)|\leq C_T|y_1-y_2|^{\gamma_1};\label{A31}
\end{eqnarray}
\begin{eqnarray}
\sup_{t\in[0,T],x\in\RR^n,\mu\in\mathscr{P}_2}\|\partial_x b(t, x, \mu, y_1)-\partial_x b(t, x, \mu, y_2)\|\leq C_T|y_1-y_2|^{\gamma_1};\label{A32}
\end{eqnarray}
\begin{eqnarray}
\sup_{t\in[0,T],x\in\RR^n,\mu\in\mathscr{P}_2}\|\partial_{\mu} b(t, x, \mu, y_1)-\partial_{\mu} b(t, x, \mu, y_2)\|_{L^2(\mu)}\leq C_T|y_1-y_2|^{\gamma_1};\label{A33}
\end{eqnarray}
\begin{eqnarray}
\sup_{t\in[0,T],x\in\RR^n,\mu\in\mathscr{P}_2}\|\partial_{y} b(t, x, \mu, y_1)-\partial_{y} b(t, x, \mu, y_2)\|\leq C_T|y_1-y_2|^{\gamma_1}.\label{A34}
\end{eqnarray}
Furthermore, if $b$ is replaced by $f$ and $g$, the properties \eref{A31}-\eref{A34} also hold.
\end{conditionA}

\smallskip
\noindent
\begin{conditionA}\label{A3} (Conditions on second-order partial derivatives)

The second-order partial derivatives $\partial^2_{xx}b(t,x,\mu,y)$, $\partial^2_{xy}b(t,x,\mu,y)$ and $\partial^2_{yy}b(t,x,\mu,y)$ exist $(t,x,y,\mu)\in[0,\infty)\times\RR^n\times\RR^m\times\mathscr{P}_2$, and $b(t,x,\cdot,y)\in C^{1,1}(\mathscr{P}_2,\RR^n)$. Moreover, $\partial^2_{xy}b(t,x,\mu,y)$, $\partial^2_{yy}b(t,x,\mu,y)$ are uniformly bounded and
\begin{eqnarray}
\sup_{t\in[0,T],x\in\RR^n,\mu\in\mathscr{P}_2}\|\partial^2_{xx} b(t, x, \mu, y_1)-\partial^2_{xx} b(t, x, \mu, y_2)\|\leq C_T|y_1-y_2|^{\gamma_2};\label{A41}
\end{eqnarray}
\begin{eqnarray}
\sup_{t\in[0,T],x\in\RR^n,\mu\in\mathscr{P}_2}\|\partial^2_{xy} b(t, x, \mu, y_1)-\partial^2_{xy} b(t, x, \mu, y_2)\|\leq C_T|y_1-y_2|^{\gamma_2};\label{A42}
\end{eqnarray}
\begin{eqnarray}
\sup_{t\in[0,T],x\in\RR^n,\mu\in\mathscr{P}_2}\|\partial^2_{yy} b(t, x, \mu, y_1)-\partial^2_{yy} b(t, x, \mu, y_2)\|\leq C_T|y_1-y_2|^{\gamma_2};\label{A431}
\end{eqnarray}
\begin{eqnarray}
\sup_{t\in[0,T],x\in\RR^n,\mu\in\mathscr{P}_2}\|\partial_{z}\partial_{\mu} b(t, x, \mu, y_1)-\partial_{z}\partial_{\mu} b(t, x, \mu, y_2)\|_{L^2(\mu)}\leq C_T|y_1-y_2|^{\gamma_2}.\label{A44}
\end{eqnarray}
Furthermore, if $b$ is replaced by $f$ and $g$, the properties \eref{A41}-\eref{A44} also hold, and
\begin{eqnarray*}
\!\!\!\!\!\!\!\!\sup_{t\in[0,T],x\in\RR^n, \mu\in\mathscr{P}_2,y\in\RR^m}\!\!\!\!\!\!\!\!\!\!\max&&\!\!\!\!\!\!\!\!\Big\{\|\partial^2_{xx}f(t,x,\mu,y)\|, \|\partial^2_{xy}f(t,x,\mu,y)\|,\|\partial^2_{yy}f(t,x,\mu,y)\|,\\
&&\!\!\!\!\|\partial^2_{xx}g(t,x,\mu,y)\|,\|\partial^2_{xy}g(t,x,\mu,y)\|,\|\partial^2_{yy}g(t,x,\mu,y)\|,\\
&&\!\!\!\!\|\partial_z\partial_{\mu}f(t,x,\mu,y)\|_{L^2(\mu)},\|\partial_z\partial_{\mu}g(t,x,\mu,y)\|_{L^2(\mu)}\Big\}\leq C_T.
\end{eqnarray*}
\end{conditionA}

\begin{remark}\label{R2.1} We here give some comments on the conditions above.
 \begin{itemize}
\item{ Conditions \eref{A11} and \eref{A21} imply that for any $T>0$, there exists $C_T>0$ such that for any $x\in\RR^n, y\in\RR^m$, $\mu\in\mathscr{P}_2$, $t\in[0,T]$,
\begin{eqnarray}
|b(t, x, \mu, y)|+\|\sigma(t, x, \mu)\|\leq C_T\left\{1+|x|+|y|+[\mu(|\cdot|^2)]^{1/2}\right\}\label{RE1}
\end{eqnarray}
and
\begin{eqnarray}
 |f(t, x,\mu,y)|+\|g(t,x,\mu, y)\|\leq C_T\left\{1+|x|+|y|+[\mu(|\cdot|^2)]^{1/2}\right\}.\label{RE2}
\end{eqnarray}}

\item{ Conditions \eref{A21} and \eref{sm} imply that for any $T>0$, there exists $C_T>0$ such that for any $x\in\RR^n, y\in\RR^m$, $\mu\in\mathscr{P}_2$, $t\in[0,T]$,
\begin{eqnarray}
 2\langle f(t, x,\mu,y), y\rangle+3\|g(t,x,\mu, y)\|^2\leq \frac{-\beta}{2}|y|^2+C_T\left\{1+|x|+[\mu(|\cdot|^2)]^{1/2}\right\}.\label{RE3}
\end{eqnarray}

}

\item{ Condition \eref{sm} is used to guarantee the existence and uniqueness of an invariant measure for the frozen equation (see Eq. (\ref{FEQ1}) below) and the solution of system \eref{Equation} has finite fourth moment.
}

\item{Using the time discretization approach, to prove the strong convergence order we need assumptions \ref{A1} and \ref{A2}. However, if using the technique of Poisson equation to prove the strong convergence order,  we needs the assumption \ref{A3} additionally. }
\end{itemize}
\end{remark}

The following theorem is the existence and uniqueness of strong solutions for system \eref{Equation}, which can be obtained by using the result due to Wang  in \cite{WFY} and whose detailed proof will be presented in the Appendix.
\begin{theorem}\label{main}
Suppose that conditions \eref{A11} and \eref{A21}  hold. For any $\ep>0$, any given initial value $x\in\RR^n, y\in \RR^m$, there exists a unique solution $\{(X^{\ep}_t,Y^{\ep}_t), t\geq 0\}$ to system \eref{Equation} and for all $T>0$,
$(X^{\ep},Y^{\ep})\in C([0,T]; \RR^n)\times C([0,T]; \RR^m), \PP-a.s.$ and
\begin{equation}\left\{\begin{array}{l}\label{mild solution}
\displaystyle
X^{\ep}_t=x+\int^t_0b(s, X^{\ep}_s, \mathscr{L}_{X^{\ep}_s}, Y^{\ep}_s)ds+\int^t_0\sigma(s, X^{\ep}_s,\mathscr{L}_{X^{\ep}_s})d W^{1}_s,\\
\displaystyle
Y^{\ep}_t=y+\frac{1}{\ep} \int^t_0f(s, X^{\ep}_s,\mathscr{L}_{X^{\ep}_s},Y^{\ep}_s)ds
+\frac{1}{\sqrt{\ep}} \int^t_0 g(s, X^{\ep}_s,\mathscr{L}_{X^{\ep}_s}, Y^{\ep}_s)dW^{2}_s.
\end{array}\right.
\end{equation}
\end{theorem}

\vskip 0.3cm

Now we formulate our first main result.
\begin{theorem}\label{main result 1}
Suppose that assumptions \ref{A1} and \ref{A2} hold. Then for any $x\in\RR^n, y\in\RR^m$ and $T>0$, we have
\begin{eqnarray}
\sup_{t\in[0,T]}\mathbb{E}|X_{t}^{\ep}-\bar{X}_{t}|^2\leq C\ep^{2/3}, \label{R1}
\end{eqnarray}
where $C$ is a constant depending on $T, |x|, |y|$. Furthermore, if there is no noise in the slow equation (i.e., $\sigma\equiv 0$), we have
\begin{eqnarray}
\sup_{t\in[0,T]}\mathbb{E}|X_{t}^{\ep}-\bar{X}_{t}|^{2}\leq C\ep. \label{R2}
\end{eqnarray}
Here $\bar{X}$ is the solution of the following averaged equation,
\begin{equation}\left\{\begin{array}{l}
\displaystyle d\bar{X}_{t}=\bar{b}(t,\bar{X}_t,\mathscr{L}_{\bar{X}_t})dt+\sigma(t,\bar{X}_t,\mathscr{L}_{\bar{X}_t})d W^{1}_t,\\
\bar{X}_{0}=x,\end{array}\right. \label{1.3}
\end{equation}
where $\bar{b}(t, x,\mu)=\int_{\RR^m}b(t, x,\mu, y)\nu^{t, x, \mu}(dy)$ and $\nu^{t, x, \mu}$ denotes the unique invariant measure for the transition semigroup of the following frozen equation:
\begin{equation}\left\{\begin{array}{l}\label{FEQ1}
\displaystyle
dY_{s}=f(t, x, \mu, Y_{s})ds+g(t, x, \mu, Y_{s})d\tilde {W}_{s}^{2},\\
Y_{0}=y,\\
\end{array}\right.
\end{equation}
where $\{\tilde{W}_{s}^{2}\}_{s\geq 0}$ is a $d_2$-dimensional Brownian motion on another complete probability space $(\tilde{\Omega}, \tilde{\mathscr{F}}, \tilde{\mathbb{P}})$.
\end{theorem}

\begin{remark}
The estimates \eref{R1} and \eref{R2} imply that the slow component $X^{\ep}_t$ strongly converges to the solution $\bar{X}_t$ of the corresponding averaged equation with convergence order $\ep^{1/3}$ and $\ep^{1/2}$ respectively. Usually, the convergence order $\ep^{1/2}$ should be optimal. Hence, under more regularity conditions on the coefficients, we will use the technique of Poisson equation to obtain the optimal convergence order in the general case (i.e., $\sigma\neq 0$), which is stated in the following theorem.
\end{remark}

\begin{theorem}\label{main result 2}
Suppose that assumptions  \ref{A1}- \ref{A3} hold. Then for any $x\in\RR^n, y\in\RR^m$ and $T>0$, we have
\begin{eqnarray}
\sup_{t\in[0,T]}\mathbb{E}|X_{t}^{\ep}-\bar{X}_{t}|^2\leq C\ep, \label{R3}
\end{eqnarray}
where $C$ is a constant depending on $T, |x|, |y|$, and $\bar{X}$ is the solution of the corresponding averaged equation \eref{1.3}.
\end{theorem}

\section{Proof of Theorem \ref{main result 1}}

In this section, we intend to use the approach of time discretization to get the strong convergence order. The proof consists of four parts, each of which is presented in the respective subsection below. In the Subsection 3.1, we give some a-priori estimates of the solution $(X^{\ep}_t, Y^{\ep}_t)$.
In the Subsection 3.2, we introduce an auxiliary process $(\hat{X}_{t}^{\ep},\hat{Y}_{t}^{\ep})$, and obtain the convergence rate of the difference process $X^{\ep}_t-\hat{X}_{t}^{\ep}$. We study the frozen equation, and prove the exponential ergodicity of the corresponding semigroup in Subsection 3.3. In the final subsection, we  prove a crucial estimate for $\sup_{t\in[0,T]}\EE|\hat{X}_{t}^{\ep}-\bar{X}_{t}|$ which relies on somewhat delicate arguments. Note that we always assume conditions \ref{A1} and \ref{A2} to hold, and the initial values $x\in\RR^n, y\in \RR^m$ are fixed in this section.

\vskip 0.2cm

 \subsection{Some a-priori estimates for  $(X^{\ep}_t, Y^{\ep}_t)$}

Firstly, we prove some uniform bounds $w.r.t.$ $\ep \in (0,1)$ for the $4$th moment of the solution $(X_{t}^{\ep}, Y_{t}^{\ep})$
to system (\ref{Equation}).
\begin{lemma} \label{PMY}
For any $T>0$, there exists  a constant $ C_T>0$ such that
\begin{eqnarray*}
\sup_{\ep\in(0,1)}\sup_{t\in [0, T]}\mathbb{E}|X_{t}^{\ep}|^{4}\leq C_{T}(1+|x|^{4}+|y|^4)\label{X}
\end{eqnarray*}
and
\begin{eqnarray*}
\sup_{\ep\in(0,1)}\sup_{t\in [0, T]}\mathbb{E}|Y_{t}^{\ep}|^{4}\leq C_{T}(1+|x|^{4}+|y|^{4}).\label{Y}
\end{eqnarray*}
\end{lemma}
\begin{proof}
By It\^{o}'s formula and estimate \eref{RE1}, we obtain for any $t\in [0, T]$,
\begin{eqnarray*}
|X_{t}^{\ep}|^{4}=\!\!\!\!\!\!\!\!&&|x|^{4}+4\int_{0} ^{t}|X_{s}^{\ep}|^{2}\langle X_{s}^{\ep}, b(s, X_{s}^{\ep}, \mathscr{L}_{X_{s}^{\ep}}, Y_{s}^{\ep})\rangle ds+4\int_{0} ^{t}|X_{s}^{\ep}|^{2}\langle X_{s}^{\ep}, \sigma(s,X_{s}^{\ep},\mathscr{L}_{X_{s}^{\ep}})dW^{1}_s\rangle\\
&&+4\int_{0} ^{t}|\langle X_{s}^{\ep}, \sigma(s,X_{s}^{\ep}, \mathscr{L}_{X_{s}^{\ep}} )\rangle|^2ds
+2\int_{0} ^{t}|X_{s}^{\ep}|^{2}\|\sigma(s,X_{s}^{\ep}, \mathscr{L}_{X_{s}^{\ep}} )\|^2ds\\
\leq\!\!\!\!\!\!\!\!&&|x|^{4}+C_T\int_{0} ^{t}(1+|X^{\ep}_s|^{4}+|Y^{\ep}_s|^{4}+[\mathscr{L}_{X_{s}^{\ep}}(|\cdot|^2)]^2)ds+4\int_{0} ^{t}|X_{s}^{\ep}|^{2}\langle X_{s}^{\ep}, \sigma(s,X_{s}^{\ep},\mathscr{L}_{X_{s}^{\ep}})dW^{1}_s\rangle.
\end{eqnarray*}
Note that $\mathscr{L}_{X_{s}^{\ep}}(|\cdot|^2)=\EE|X_{s}^{\ep}|^2$. Hence, we have
\begin{eqnarray}
\sup_{t\in[0,T]}\EE|X_{t}^{\ep}|^4
\leq\!\!\!\!\!\!\!\!&&C_{T}(|x|^{4}+1)+C_T\int^T_0 \EE|Y_{t}^{\ep}|^{4} dt+C_T\int^T_0 \EE|X_{t}^{\ep}|^{4}dt.\label{I3.8}
\end{eqnarray}

Using It\^{o} formula again and taking expectation, we get
\begin{eqnarray*}
\mathbb{E}|Y_{t}^{\ep}|^{4}=\!\!\!\!\!\!\!\!&&|y|^{4}+\frac{4}{\ep}\int_{0} ^{t}\mathbb{E}\left[|Y_{s}^{\ep}|^{2}\langle f(s, X_{s}^{\ep},\mathscr{L}_{X_{s}^{\ep}},Y_{s}^{\ep}),Y_{s}^{\ep}\rangle\right] ds\\
&&+\frac{2}{\ep}\int_{0} ^{t}\mathbb{E}\left[|Y_{s}^{\ep}|^{2}\|g(s, X_{s}^{\ep}, \mathscr{L}_{X_{s}^{\ep}}, Y_{s}^{\ep})\|^2\right] ds+\frac{4}{\ep}\int_{0} ^{t}\mathbb{E}|\langle Y_{s}^{\ep}, g(s, X_{s}^{\ep}, \mathscr{L}_{X_{s}^{\ep}}, Y_{s}^{\ep})\rangle|^2 ds.
\end{eqnarray*}
By \eref{RE3}, there exists $\beta>0$ such that for any $t\in [0, T]$,
\begin{eqnarray*}
\frac{d}{dt}\mathbb{E}|Y_{t}^{\ep}|^{4}\leq\!\!\!\!\!\!\!\!&&\frac{1}{\ep}\EE\left[4|Y_{t}^{\ep}|^{2}\langle f(t, X_{t}^{\ep},\mathscr{L}_{X_{t}^{\ep}},Y_{t}^{\ep}),Y_{t}^{\ep}\rangle+6|Y_{t}^{\ep}|^{2}\|g(t, X_{t}^{\ep},\mathscr{L}_{X_{t}^{\ep}},Y_{t}^{\ep})\|^2\right]\\
 \leq\!\!\!\!\!\!\!\!&&-\frac{\beta}{\ep}\mathbb{E}|Y_{t}^{\ep}|^{4}
+\frac{C_T}{\ep}\left(\EE|X_{t}^{\ep}|^{4}+1\right).
\end{eqnarray*}
The comparison theorem implies
\begin{eqnarray}
\mathbb{E}|Y_{t}^{\ep}|^{4}\leq\!\!\!\!\!\!\!\!&&|y|^{4}e^{-\frac{\beta t}{\ep}}+\frac{C_T}{\ep}\int^t_0 e^{-\frac{\beta(t-s)}{\ep}}\left(\EE|X_{s}^{\ep}|^{4}+1\right)ds\nonumber\\
\leq\!\!\!\!\!\!\!\!&&|y|^{4}+C_T\left( \sup_{s\in [0,t]}\EE |X_s^{\ep}|^4+1\right).\label{I3.9}
\end{eqnarray}
This and \eref{I3.8} yield
\begin{eqnarray*}
\sup_{t\in[0,T]}\EE|X_{t}^{\ep}|^{4}\leq C_{T}(|x|^{4}+|y|^{4}+1)+C_{T}\int^T_0 \sup_{s\in[0,t]}\EE|X_{s}^{\ep}|^{4}dt.
\end{eqnarray*}
Then by Grownall's inequality, we finally obtain
\begin{eqnarray*}
\sup_{t\in[0,T]}\EE|X_{t}^{\ep}|^{4}\leq C_{T}(|x|^{4}+|y|^{4}+1),
\end{eqnarray*}
which also gives
\begin{eqnarray*}
\sup_{t\in[0, T]}\mathbb{E}|Y_{t}^{\ep}|^{4}\leq C_{T}(|x|^{4}+|y|^{4}+1).
\end{eqnarray*}
 The proof is complete.
\end{proof}

\begin{lemma} \label{COX} For any $T>0$, $0\leq t\leq t+h\leq T$ and $\ep\in(0,1)$, there exists a constant $C_{T}>0$ such that
\begin{eqnarray*}
\mathbb{E}|X_{t+h}^{\ep}-X_{t}^{\ep}|^{2}\leq C_{T}(1+|x|^2+|y|^2)h.
\end{eqnarray*}
\end{lemma}

\begin{proof}
It is easy to see that
\begin{eqnarray}
X_{t+h}^{\ep}-X_{t}^{\ep}=\int^{t+h}_t b(s, X^{\ep}_s, \mathscr{L}_{X_{s}^{\ep}}, Y^{\ep}_s)ds+\int^{t+h}_t\sigma(s, X^{\ep}_s, \mathscr{L}_{X_{s}^{\ep}})dW^1_s.\nonumber
\end{eqnarray}
Then by estimate \eref{RE1} and Lemma \ref{PMY}, we obtain
\begin{eqnarray*}
\EE|X_{t+h}^{\ep}-X_{t}^{\ep}|^2\leq &&\!\!\!\!\!\!\!\!C\EE\left|\int^{t+h}_t b(s, X^{\ep}_s, \mathscr{L}_{X_{s}^{\ep}},Y^{\ep}_s)ds\right|^2+C\EE\left|\int^{t+h}_t\sigma(s, X^{\ep}_s,\mathscr{L}_{X_{s}^{\ep}})dW^1_s\right|^2\\
\leq&& \!\!\!\!\!\!\!\!C \EE \left|\int^{t+h}_t|b(s, X^{\ep}_s, \mathscr{L}_{X_{s}^{\ep}},Y^{\ep}_s)| ds\right|^2+C\int^{t+h}_t \EE\|\sigma(s, X^{\ep}_s, \mathscr{L}_{X_{s}^{\ep}})\|^2ds\\
\leq &&\!\!\!\!\!\!\!\! C_T h\EE\int^{t+h}_t \!\!(1+|X^{\ep}_s|^{2}+| Y^{\ep}_s|^{2}+\EE|X_{s}^{\ep}|^2) ds\!+\!C_T\int^{t+h}_t\!\! \EE(1+|X^{\ep}_s|^2+\EE|X_{s}^{\ep}|^2)ds\\
\leq &&\!\!\!\!\!\!\!\! C_{T}(1+|x|^2+|y|^2)h.
\end{eqnarray*}
The proof is complete.
\end{proof}

\vskip 0.2cm

\subsection{ Estimates for the auxiliary process $(\hat{X}_{t}^{\ep},\hat{Y}_{t}^{\ep})$}

Following the idea of Khasminskii in \cite{K1},
we introduce an auxiliary process $(\hat{X}_{t}^{\ep},\hat{Y}_{t}^{\ep})\in \RR^n\times \RR^m$
and divide $[0,T]$ into intervals of size $\delta$, where $\delta$ is a fixed positive number depending on $\ep$, which will be chosen later.
We construct a process $\hat{Y}_{t}^{\ep}$ with initial value $\hat{Y}_{0}^{\ep}=Y^{\ep}_{0}=y$ such that for $t\in[k\delta,\min((k+1)\delta,T)]$,
\begin{eqnarray*}
\hat{Y}_{t}^{\ep}=\hat Y_{k\delta}^{\ep}+\frac{1}{\ep}\int_{k\delta}^{t}
f(k\delta, X_{k\delta}^{\ep},\mathscr{L}_{X_{k\delta}^{\ep}},\hat{Y}_{s}^{\ep})ds+\frac{1}{\sqrt{\ep}}\int_{k\delta}^{t}g(k\delta, X_{k\delta}^{\ep},\mathscr{L}_{X_{k\delta}^{\ep}},\hat{Y}_{s}^{\ep})dW^{2}_s,
\end{eqnarray*}
i.e.,
\begin{eqnarray*}
\hat{Y}_{t}^{\ep}=y+\frac{1}{\ep}\int_{0}^{t} f(s(\delta), X_{s(\delta)}^{\ep},\mathscr{L}_{X_{s(\delta)}^{\ep}},\hat{Y}_{s}^{\ep})ds+\frac{1}{\sqrt{\ep}}\int_{0}^{t}g(s(\delta), X_{s(\delta)}^{\ep},\mathscr{L}_{X_{s(\delta)}^{\ep}},\hat{Y}_{s}^{\ep})dW^{2}_s,
\end{eqnarray*}
where $s(\delta)=[{s}/{\delta}]\delta$, and $[{s}/{\delta}]$ is the integer part of ${s}/{\delta}$.  Also, we define the process $\hat{X}_{t}^{\ep}$ by
$$
\hat{X}_{t}^{\ep}=x+\int_{0}^{t}b(s(\delta), X_{s(\delta)}^{\ep}, \mathscr{L}_{X_{s(\delta)}^{\ep}},\hat{Y}_{s}^{\ep})ds+\int_{0}^{t}\sigma(s, X^{\ep}_s, \mathscr{L}_{X_{s}^{\ep}})W^{1}_s.
$$

By the construction of $\hat{Y}_{t}^{\ep}$ and by similar argument as in the proof of Lemma \ref{PMY}, it is easy to obtain the following estimates we omit whose proof here.
\begin{lemma} \label{MDY}
For any $T>0$, there exists  a constant $C_{T}>0$ such that
\begin{eqnarray*}
\sup_{\ep\in(0,1)}\sup_{t\in [0, T]}\mathbb{E}|\hat Y_{t}^{\ep}|^{4}\leq C_{T}(1+|x|^{4}+|y|^{4}).\label{Y}
\end{eqnarray*}
\end{lemma}

\vskip 0.2cm
Now, we intend to estimate the difference process $Y_{t}^{\ep}-\hat{Y}_{t}^{\ep}$ and furthermore the difference process $X^{\ep}_t-\hat{X}_{t}^{\ep}$.
\begin{lemma} \label{DEY} For any $T>0$, there exists a constant $C_{T}>0$ such that
$$
\sup_{\ep\in(0,1)}\sup_{t\in[0,T]}\mathbb{E}|Y_{t}^{\ep}-\hat{Y}_{t}^{\ep}|^{2}\leq C_{T}(1+|x|^2+|y|^2)\delta.
$$
\end{lemma}
\begin{proof}
Note that
\begin{eqnarray*}
Y_{t}^{\ep}-\hat{Y}_{t}^{\ep}=\!\!\!\!\!\!\!\!&&\frac{1}{\ep}\int_{0}^{t} \left[f(s, X_{s}^{\ep}, \mathscr{L}_{X_{s}^{\ep}}, Y_{s}^{\ep})-f(s(\delta), X_{s(\delta)}^{\ep},\mathscr{L}_{X_{s(\delta)}^{\ep}},\hat{Y}_{s}^{\ep})\right]ds\\
&&+\frac{1}{\sqrt{\ep}}\int_{0}^{t}\left[g(s, X_{s}^{\ep},\mathscr{L}_{X_{s}^{\ep}}, Y_{s}^{\ep})-g(s(\delta), X_{s(\delta)}^{\ep},\mathscr{L}_{X_{s(\delta)}^{\ep}},\hat{Y}_{s}^{\ep})\right]dW^{2}_s.
\end{eqnarray*}

By It\^{o}'s formula, we have for any $t\in[0,T]$,
\begin{eqnarray*}
&&\mathbb{E}|Y_{t}^{\ep}-\hat{Y}_{t}^{\ep}|^{2}\\
=\!\!\!\!\!\!\!\!&&\frac{1}{\ep}\int^t_0\mathbb{E}\Big[2\langle f(s,X_{s}^{\ep},\mathscr{L}_{X_{s}^{\ep}}, Y_{s}^{\ep})-f(s(\delta), X_{s(\delta)}^{\ep},\mathscr{L}_{X_{s(\delta)}^{\ep}},\hat{Y}_{s}^{\ep}), Y_{s}^{\ep}-\hat{Y}_{s}^{\ep}\rangle\Big]ds \nonumber\\
&&+\frac{1}{\ep}\int^t_0\EE\|g(s, X_{s}^{\ep},\mathscr{L}_{X_{s}^{\ep}},Y_{s}^{\ep})-g(s(\delta), X_{s(\delta)}^{\ep},\mathscr{L}_{X_{s(\delta)}^{\ep}},\hat{Y}_{s}^{\ep})\|^2ds\nonumber\\
\leq\!\!\!\!\!\!\!\!&&\frac{1}{\ep}\int^t_0\mathbb{E}\Big[2\langle f(s, X_{s}^{\ep},\mathscr{L}_{X_{s}^{\ep}},Y_{s}^{\ep})-f(s, X_{s}^{\ep},\mathscr{L}_{X_{s}^{\ep}},\hat{Y}_{s}^{\ep}), Y_{s}^{\ep}-\hat{Y}_{s}^{\ep}\rangle\\
&&\quad\quad+3\|g(s, X_{s}^{\ep},\mathscr{L}_{X_{s}^{\ep}},Y_{s}^{\ep})-g(s, X_{s}^{\ep},\mathscr{L}_{X_{s}^{\ep}},\hat{Y}_{s}^{\ep})\|^2\Big]ds\\
&&+\frac{2}{\ep}\int^t_0\EE\left[\langle f(s, X_{s}^{\ep},\mathscr{L}_{X_{s}^{\ep}},\hat{Y}_{s}^{\ep})-f(s(\delta), X_{s(\delta)}^{\ep},\mathscr{L}_{X_{s(\delta)}^{\ep}},\hat{Y}_{s}^{\ep}), Y_{s}^{\ep}-\hat{Y}_{s}^{\ep}\rangle\right] ds \nonumber\\
&&+\frac{1}{3\ep}\EE\|g(s, X_{s}^{\ep},\mathscr{L}_{X_{s}^{\ep}}, \hat{Y}_{s}^{\ep})-g(s(\delta), X_{s(\delta)}^{\ep},\mathscr{L}_{X_{s(\delta)}^{\ep}},\hat{Y}_{s}^{\ep})\|^2ds.
\end{eqnarray*}
Then using the following estimate
$$\mathbb{W}_2(\mathscr{L}_{X_{s}^{\ep}},\mathscr{L}_{X_{s(\delta)}^{\ep}})^2\leq \EE|X_{s}^{\ep}-X_{s(\delta)}^{\ep}|^2$$
and the conditions \eref{A21},  \eref{sm}, there exists $\beta>0$ such that for any $t\in [0, T]$,
\begin{eqnarray*}
\frac{d}{dt}\mathbb{E}|Y_{t}^{\ep}-\hat{Y}_{t}^{\ep}|^{2}
\leq\!\!\!\!\!\!\!\!&&\frac{-\beta}{\ep}\mathbb{E}|Y_{t}^{\ep}-\hat{Y}_{t}^{\ep}|^2+\frac{C_T}{\ep}\EE\left[\delta^2+|X_{t}^{\ep}-X_{t(\delta)}^{\ep}|^2+\mathbb{W}_2(\mathscr{L}_{X_{t}^{\ep}},\mathscr{L}_{X_{t(\delta)}^{\ep}})^2\right]\nonumber\\
\leq\!\!\!\!\!\!\!\!&&
-\frac{\beta}{\ep}\mathbb{E}|Y_{t}^{\ep}-\hat{Y}_{t}^{\ep}|^{2}+\frac{C_T}{\ep}\mathbb{E}|X_{t}^{\ep}-X_{t(\delta)}^{\ep}|^2+\frac{C_T\delta^2}{\ep}.\nonumber
\end{eqnarray*}
Finally, the comparison theorem and Lemma \ref{COX} yield
\begin{eqnarray*}
\mathbb{E}|Y_{t}^{\ep}-\hat{Y}_{t}^{\ep}|^{2}\leq\!\!\!\!\!\!\!\!&& \frac{C_T}{\ep}\int^t_0 e^{-\frac{\beta(t-s)}{\ep}}\mathbb{E}|X_{s}^{\ep}-X_{s(\delta)}^{\ep}|^2ds+\frac{C_T\delta^2}{\ep}\int^t_0 e^{-\frac{\beta(t-s)}{\ep}}ds\\
\leq\!\!\!\!\!\!\!\!&&C_{T}(1+|x|^2+|y|^2)\delta.
\end{eqnarray*}
The proof is complete.
\end{proof}

\begin{lemma} \label{DEX} For any $T>0$, there exists a constant $C_{T}>0$ such that
\begin{eqnarray*}
\sup_{t\in[0,T]}\mathbb{E}|X_{t}^{\ep}-\hat{X}_{t}^{\ep}|^2\leq C_{T}(1+|x|^2+|y|^2)\delta.
\end{eqnarray*}
\end{lemma}

\begin{proof}
Recall that
\begin{eqnarray*}
X^{\ep}_t=x+\int^t_0 b(s, X^{\ep}_s, \mathscr{L}_{X^{\ep}_s},Y^{\ep}_s)ds+\int^t_0\sigma(s, X^{\ep}_s,\mathscr{L}_{X^{\ep}_s})dW^{1}_s
\end{eqnarray*}
and that
\begin{eqnarray*}
\hat{X}^{\ep}_t=x+\int^t_0 b(s(\delta), X^{\ep}_{s(\delta)}, \mathscr{L}_{X^{\ep}_{s(\delta)}},\hat{Y}^{\ep}_s)ds+\int^t_0\sigma(s, X^{\ep}_s,\mathscr{L}_{X^{\ep}_s})dW^{1}_s.
\end{eqnarray*}
Then we have
\begin{eqnarray*}
X^{\ep}_t-\hat{X}^{\ep}_t=\int^t_0\big[b(s, X^{\ep}_s, \mathscr{L}_{X^{\ep}_s}, Y^{\ep}_s)-b(s(\delta), X^{\ep}_{s(\delta)}, \mathscr{L}_{X^{\ep}_{s(\delta)}},\hat{Y}^{\ep}_s)\big]ds.
\end{eqnarray*}
By Lemmas \ref{COX} and \ref{DEY}, we obtain
\begin{eqnarray*}
\sup_{t\in[0,T]}\EE|X^{\ep}_{t} -\hat{X}^{\ep}_{t}|^2\leq\!\!\!\!\!\!\!\!&&\EE\left[\int^{T}_0\left|b(s, X^{\ep}_s, \mathscr{L}_{X^{\ep}_s},Y^{\ep}_s)-b(s(\delta), X^{\ep}_{s(\delta)}, \mathscr{L}_{X^{\ep}_{s(\delta)}},\hat{Y}^{\ep}_s)\right| ds\right]^2\nonumber\\
\leq\!\!\!\!\!\!\!\!&& C_{T}\EE\int^{T}_0 (\delta^2+|X^{\ep}_s-X^{\ep}_{s(\delta)}|^2+\mathbb{W}_2(\mathscr{L}_{X^{\ep}_s},\mathscr{L}_{X^{\ep}_{s(\delta)}})^2+|Y^{\ep}_s-\hat Y^{\ep}_s|^2)ds\\
\leq\!\!\!\!\!\!\!\!&& C_{T}(1+|x|^2+|y|^2)\delta. \label{bs}
\end{eqnarray*}
The proof is complete.
\end{proof}

\subsection{The frozen equation}
We first introduce the frozen equation associated to the fast motion for fixed $t\geq 0, x\in \RR^n$ and $\mu\in\mathscr{P}_2$,
 \begin{equation}\left\{\begin{array}{l}\label{FEQ}
\displaystyle
dY_{s}=f(t, x,\mu,Y_{s})dt+g(t, x,\mu,Y_{s})d\tilde{W}_{s}^{2},\\
Y_{0}=y,\\
\end{array}\right.
\end{equation}
where $\{\tilde {W}_{s}^{2}\}_{s\geq 0}$ is a $d_2$-dimensional Brownian motion on another complete probability space $(\tilde{\Omega}, \tilde{\mathscr{F}}, \tilde{\mathbb{P}})$ and $\{\tilde{\mathscr{F}}_{t},t\geq 0\}$ is the natural filtration generated by $\tilde{W}_{t}^{2}$.

Under the conditions \eref{A21} and \eref{sm}, it is easy to prove for any initial data $y\in \RR^m$ that
Eq.$(\ref{FEQ})$ has a unique strong solution $\{Y_{s}^{t,x,\mu,y}\}_{s\geq 0}$, which is a homogeneous Markov process. Moreover, for any $t\in[0,T]$, $\sup_{s\geq 0}\tilde \EE|Y_{s}^{t,x,\mu,y}|^2\leq C_T\left[1+|x|^2+|y|^2+\mu(|\cdot|^2)\right]$.

Let $\{P^{t,x,\mu}_s\}_{s\geq 0}$ be the transition semigroup of $Y_{s}^{t,x,\mu,y}$, i.e., for any bounded measurable function $\varphi:\RR^m\rightarrow \mathbb{R}$,
$$
P^{t,x,\mu}_s\varphi(y):=\tilde \EE\varphi(Y_{s}^{t,x,\mu,y}), \quad y\in\RR^m, s\geq 0,
$$
where $\tilde \EE$ is the expectation on $(\tilde{\Omega}, \tilde{\mathscr{F}}, \tilde{\mathbb{P}})$. Then e.g. by \cite[Theorem 4.3.9]{LR1}, under the assumption \ref{A1}, it is easy to see that $P^{t,x,\mu}_s$ has a unique invariant measure $\nu^{t,x,\mu}$ satisfying
$$\int_{\RR^m}|y|\nu^{t, x,\mu}(dy)\leq C_T\left\{1+|x|+[\mu(|\cdot|^2)]^{1/2}\right\}.$$

\begin{lemma}\label{L3.6}
For any $T>0$, $s>0$, $ t_i\in[0, T]$, $x_i\in\RR^n$, $\mu_i\in\mathscr{P}_2$ and $y_i\in\RR^m$, $i=1,2$, we have
\begin{eqnarray*}
\tilde{\mathbb{E}}|Y^{t_1, x_1,\mu_1,y_1}_s-Y^{t_2, x_2,\mu_2,y_2}_s|^2\leq e^{-\beta s}|y_1-y_2|^2+C_T\left[|t_1-t_2|^2+|x_1-x_2|^2+\mathbb{W}_2(\mu_1,\mu_2)^2\right].
\end{eqnarray*}
\end{lemma}
\begin{proof}
Note that
\begin{eqnarray*}
Y^{t_1,x_1,\mu_1,y_1}_s-Y^{t_2,x_2,\mu_2,y_2}_s=\!\!\!\!\!\!\!\!&&y_1-y_2+\int^s_0 f(t_1,x_1, \mu_1, Y^{t_1,x_1,\mu_1,y_1}_r)-f(t_2, x_2, \mu_2,Y^{t_2, x_2,\mu_2,y_2}_r)dr\\
&&+\int^s_0 g(t_1, x_1, \mu_1,Y^{t_1, x_1,\mu_1,y_1}_r)-g(t_2, x_2, \mu_2,Y^{t_2, x_2,\mu_2,y_2}_r)d \tilde {W}^2_r.
\end{eqnarray*}
By It\^{o}'s formula we have
\begin{eqnarray*}
&&\tilde{\mathbb{E}}|Y^{t_1,x_1,\mu_1,y_1}_s-Y^{t_2,x_2,\mu_2,y_2}_s|^2\\
=\!\!\!\!\!\!\!\!&&\int^s_0 \tilde{\mathbb{E}}\big[2\langle f(t_1, x_1, \mu_1,Y^{t_1,x_1,\mu_1,y_1}_r)-f(t_2,x_2, \mu_2,Y^{t_2, x_2,\mu_2,y_2}_r), Y^{t_1,x_1,\mu_1,y_1}_r-Y^{t_2, x_2,\mu_2,y_2}_r\rangle\\
&&+\|g(t_1, x_1, \mu_1,Y^{t_1,x_1,\mu_1,y_1}_r)-g(t_2, x_2, \mu_2,Y^{t_2,x_2,\mu_2,y_2}_r)\|^2\big]dr.
\end{eqnarray*}
Then by Young's inequality and conditions \eref{A21} and \eref{sm}, there exists $\beta>0$ such that
\begin{eqnarray*}
&&\frac{d}{ds}\tilde{\mathbb{E}}|Y^{t_1,x_1,\mu_1,y_1}_s-Y^{t_2,x_2,\mu_2,y_2}_s|^2\\
=\!\!\!\!\!\!\!\!&&\tilde{\mathbb{E}}\big[2\langle f(t_1, x_1, \mu_1,Y^{t_1,x_1,\mu_1,y_1}_s)-f(t_2,x_2, \mu_2,Y^{t_2, x_2,\mu_2,y_2}_s), Y^{t_1,x_1,\mu_1,y_1}_s-Y^{t_2, x_2,\mu_2,y_2}_s\rangle\\
&&+\|g(t_1, x_1, \mu_1,Y^{t_1,x_1,\mu_1,y_1}_s)-g(t_2, x_2, \mu_2,Y^{t_2,x_2,\mu_2,y_2}_s)\|^2\big]\\
\leq\!\!\!\!\!\!\!\!&&\tilde{\mathbb{E}}\big[2\left\langle f(t_1, x_1, \mu_1,Y^{t_1, x_1,\mu_1,y_1}_s)-f(t_1, x_1, \mu_1,Y^{t_2,x_2,\mu_2,y_2}_s), Y^{t_1,x_1,\mu_1,y_1}_s-Y^{t_2,x_2,\mu_2,y_2}_s\right\rangle\\
&&+3\left\|g(t_1,x_1, \mu_1,Y^{t_1,x_1,\mu_1,y_1}_s)-g(t_1, x_1, \mu_1,Y^{t_2, x_2,\mu_2,y_2}_s)\right\|^2\big]\\
&&+\tilde{\mathbb{E}}\left[2\left\langle f(t_1,x_1, \mu_1,Y^{t_2,x_2,\mu_2,y_2}_s)-f(t_2,x_2, \mu_2,Y^{t_2,x_2,\mu_2,y_2}_s), Y^{t_1,x_1,\mu_1,y_1}_s-Y^{t_2,x_2,\mu_2,y_2}_s\right\rangle\right]\\
&&+\frac{1}{3}\tilde{\mathbb{E}}\left\|g(t_1,x_1, \mu_1,Y^{t_2, x_2,\mu_2,y_2}_s)-g(t_2, x_2, \mu_2,Y^{t_2, x_2,\mu_2,y_2}_s)\right\|^2\\
\leq\!\!\!\!\!\!\!\!&& -\beta\tilde{\mathbb{E}}\left|Y^{t_1,x_1,\mu_1,y_1}_s-Y^{t_2,x_2,\mu_2,y_2}_s\right|^2+C_T\left[|t_1-t_2|^2+|x_1-x_2|^2+\mathbb{W}_2(\mu_1,\mu_2)^2\right].
\end{eqnarray*}
Hence, the comparison theorem yields for any $s\geq 0$,
\begin{eqnarray*}
\tilde{\mathbb{E}}|Y^{t_1,x_1,\mu_1,y_1}_s-Y^{t_2, x_2,\mu_2,y_2}_s|^2\leq e^{-\beta s}|y_1-y_2|^2+C_T\left[|t_1-t_2|^2+|x_1-x_2|^2+\mathbb{W}_2(\mu_1,\mu_2)^2\right].
\end{eqnarray*}
The proof is complete.
\end{proof}

\begin{proposition}\label{Rem 4.1}  For any $T>0$, $t\in[0,T], x\in\RR^n$, $\mu\in\mathscr{P}_2$, $s\geq 0$ and $y\in \RR^m$,
\begin{eqnarray}
\left|\tilde \EE b(t, x, \mu, Y^{t,x,\mu,y}_s)-\bar b(t,x,\mu)\right|\leq\!\!\!\!\!\!\!\!&& C_Te^{-\frac{\beta s}{2}}\left\{1+|x|+|y|+[\mu(|\cdot|^2)]^{1/2}\right\}, \label{ergodicity}
\end{eqnarray}
where $\bar b(t,x,\mu)=\int_{\RR^m}b(t, x,\mu, z)\nu^{t,x,\mu}(dz)$.
\end{proposition}
\begin{proof}
By the definition of an invariant measure and Lemma \ref{L3.6},  for any $s\geq 0$ we have
\begin{eqnarray*}
\left|\tilde \EE b(t, x, \mu, Y^{t,x,\mu,y}_s)-\bar b(t,x,\mu)\right|=\!\!\!\!\!\!\!\!&&\left|\tilde \EE b(t, x, \mu, Y^{t,x,\mu,y}_s)-\int_{\RR^m}b(t, x,\mu, z)\nu^{t,x,\mu}(dz)\right|\\
\leq\!\!\!\!\!\!\!\!&& \left|\int_{\RR^m}\left[\tilde \EE b(t, x, \mu,Y^{t,x,\mu,y}_s)-\tilde \EE b(t, x, \mu,Y^{t,x,\mu,z}_s)\right]\nu^{t,x,\mu}(dz)\right|\\
\leq\!\!\!\!\!\!\!\!&& C_T\int_{\RR^m}\tilde \EE\left| Y^{t,x,\mu,y}_s-Y^{t,x,\mu,z}_s\right|\nu^{t,x,\mu}(dz)\\
\leq\!\!\!\!\!\!\!\!&& C_T e^{-\frac{\beta s}{2}}\int_{\RR^m}|y-z|\nu^{t,x,\mu}(dz)\\
\leq\!\!\!\!\!\!\!\!&&C_T e^{-\frac{\beta s}{2}}\left\{1+|x|+|y|+[\mu(|\cdot|^2)]^{1/2}\right\}.
\end{eqnarray*}
The proof is complete.
\end{proof}

\vskip 0.2cm
\subsection{The averaged equation}
We can introduce the averaged equation as follows,
\begin{equation}\left\{\begin{array}{l}
\displaystyle d\bar{X}_{t}=\bar{b}(t, \bar{X}_{t}, \mathscr{L}_{\bar{X}_{t}})dt+\sigma(t, \bar{X}_{t}, \mathscr{L}_{\bar{X}_{t}})dW^{1}_t,\\
\bar{X}_{0}=x\in \RR^n,\end{array}\right. \label{3.1}
\end{equation}
with
\begin{eqnarray*}
\bar{b}(t,x,\mu)=\int_{\RR^m}b(t,x,\mu,z)\nu^{t,x,\mu}(dz),
\end{eqnarray*}
where $\nu^{t, x,\mu}$ is the unique invariant measure for Eq.(\ref{FEQ}).

\vspace{0.2cm}
The following lemma gives the existence, uniqueness and uniformly estimates for the solution of Eq. \eref{3.1}, whose proof will be presented in the Appendix.

\begin{lemma} \label{PMA} For any $x\in\RR^n$, Eq.(\ref{3.1}) has a unique solution $\bar{X}_t$. Moreover, for any $T>0$, there exists a constant $C_{T}>0$ such that
\begin{eqnarray}
\sup_{t\in [0, T]}\mathbb{E}|\bar{X}_{t}|^{2}\leq C_{T}(1+|x|^{2}).\label{3.9}
\end{eqnarray}
\end{lemma}

Now, we estimate the error between the auxiliary process $\hat{X}_{t}^{\ep}$ and the solution $\bar{X}_{t}$ of the averaged equation .
\begin{lemma} \label{ESX} For any $T>0$, there exists a constant $C_{T}>0$ such that
\begin{eqnarray*}
\sup_{t\in[0,T]}\mathbb{E}|\hat{X}_{t}^{\ep}-\bar{X}_{t}|^2\leq C_{T}(1+|x|^3+|y|^3)\left(\frac{\ep}{\delta^{1/2}}+\ep+\frac{\ep^2}{\delta}+\delta\right).
\end{eqnarray*}
\end{lemma}

\begin{proof} We will divide the proof into three steps.

\vspace{0.3cm}
\textbf{Step 1.} Recall that
\begin{eqnarray*}
\hat{X}_{t}^{\ep}-\bar{X}_{t}=\!\!\!\!\!\!\!\!&&\int_{0}^{t}\left[b(s(\delta), X_{s(\delta)}^{\ep},\mathscr{L}_{X_{s(\delta)}^{\ep}},\hat{Y}_{s}^{\ep})-\bar{b}(s,\bar{X}_{s},\mathscr{L}_{\bar X_{s}})\right]ds\\
&&+\int_{0}^{t}\left[\sigma(s, X^{\ep}_{s},\mathscr{L}_{X_{s}^{\ep}})-\sigma(s, \bar{X}_{s},\mathscr{L}_{\bar X_{s}})\right]dW^{1}_s\\
=\!\!\!\!\!\!\!\!&&\int_{0}^{t}\left[b(s(\delta), X_{s(\delta)}^{\ep},\mathscr{L}_{ X_{s(\delta)}^{\ep}},\hat{Y}_{s}^{\ep})-\bar{b}(s(\delta), X^{\ep}_{s(\delta)},\mathscr{L}_{X^{\ep}_{s(\delta)}})\right]ds\\
&&+\int_{0}^{t}\left[\bar{b}(s(\delta), X^{\ep}_{s(\delta)},\mathscr{L}_{X^{\ep}_{s(\delta)}})-\bar{b}(s, X^{\ep}_{s},\mathscr{L}_{X^{\ep}_{s}})\right]ds\\
&&+\int_{0}^{t}\left[\bar{b}(s, X_{s}^{\ep},\mathscr{L}_{X^{\ep}_{s}})-\bar{b}(s, \bar{X}_s,\mathscr{L}_{\bar X_{s}})\right]ds\\
&&+\int_{0}^{t}\left[\sigma(s, X^{\ep}_{s}, \mathscr{L}_{X^{\ep}_{s}})-\sigma(s, \bar{X}_{s},\mathscr{L}_{\bar X_{s}})\right]dW^{1}_s.
\end{eqnarray*}
Then it is esay to see that for any $t\in [0, T]$, we have
\begin{eqnarray}
\EE|\hat{X}_{t}^{\ep}-\bar{X}_{t}|^2\leq\!\!\!\!\!\!\!\!&&C\EE\left|\int_{0}^{t}\left[b(s(\delta), X_{s(\delta)}^{\ep},\mathscr{L}_{ X_{s(\delta)}^{\ep}},\hat{Y}_{s}^{\ep})-\bar{b}(s(\delta), X^{\ep}_{s(\delta)},\mathscr{L}_{X^{\ep}_{s(\delta)}})\right]ds\right|^2\nonumber\\
&&+C_T\EE\int_{0}^{t}\left|\bar{b}(s(\delta), X^{\ep}_{s(\delta)},\mathscr{L}_{X^{\ep}_{s(\delta)}})-\bar{b}(s, X^{\ep}_{s},\mathscr{L}_{X^{\ep}_{s}})\right|^2ds\nonumber\\
&&+C_T\EE\int_{0}^{t}\left|\bar{b}(s, X_{s}^{\ep},\mathscr{L}_{X^{\ep}_{s}})-\bar{b}(s, \bar{X}_s,\mathscr{L}_{\bar X_{s}})\right|^2ds\nonumber\\
&&+C\EE\int_{0}^{t}\left\|\sigma(s, X^{\ep}_{s},\mathscr{L}_{X^{\ep}_{s}})-\sigma(s, \bar{X}_{s},\mathscr{L}_{\bar X_{s}})\right\|^2ds\nonumber\\
:=\!\!\!\!\!\!\!\!&&\sum^4_{i=1}I_i(t).\label{I3.14}
\end{eqnarray}

For $I_2(t)$ we have by the Lipschitz property of $\bar b(\cdot,\cdot,\cdot)$ (see \eref{4.16} below) that
\begin{eqnarray}
\sup_{t\in [0,T]}I_2(t)\leq\!\!\!\!\!\!\!\!&&C_T\delta^2+\EE\int^T_0|X^{\ep}_{s(\delta)}-X^{\ep}_{s}|^2ds\nonumber\\
\leq\!\!\!\!\!\!\!\!&&C_T(1+|x|^2+|y|^2)\delta.\label{I3.15}
\end{eqnarray}

For $I_i(t)$, $i=3,4$, Lemma \ref{DEX} implies
\begin{eqnarray}
\sup_{t\in[0,T]}I_3(t)\leq\!\!\!\!\!\!\!\!&& C_T\int_{0}^{T}\mathbb{E}|X^{\ep}_{t}-\bar X_{t}|^2dt\nonumber\\
\leq\!\!\!\!\!\!\!\!&&C_T\int_{0}^{T}\mathbb{E}|X^{\ep}_{t}-\hat X^{\ep}_{t}|^2dt+C_T\int_{0}^{T}\mathbb{E}|\hat X^{\ep}_{t}-\bar X_{t}|^2dt\nonumber\\
\leq\!\!\!\!\!\!\!\!&& C_T(1+|x|^2+|y|^2)\delta+C_T\int_{0}^{T}\mathbb{E}|\hat X^{\ep}_{t}-\bar X_{t}|^2dt.\label{I3.17}
\end{eqnarray}
Similarly, by condition \eref{A11},
\begin{eqnarray}
\sup_{t\in[0,T]}I_4(t)\leq\!\!\!\!\!\!\!\!&&C_T(1+|x|^2+|y|^2)\delta+C_T\int_{0}^{T}\mathbb{E}|\hat X^{\ep}_{t}-\bar X_{t}|^2dt.\label{I3.18}
\end{eqnarray}

Therefore, \eref{I3.14}-\eref{I3.18} yield
\begin{eqnarray}
\sup_{t\in[0, T]}\EE|\hat{X}_{t}^{\ep}-\bar{X}_{t}|^2\leq\!\!\!\!\!\!\!\!&&\sup_{t\in[0,T]}I_1(t)+C_T(1+|x|^2+|y|^2)\delta\nonumber\\
&&+C_T\int^{T}_0\EE|\hat X^{\ep}_t-\bar{X}_{t}|^2dt.\label{I3.19}
\end{eqnarray}
Then combining this with the following estimate of $I_1(t)$,
\begin{eqnarray}
\sup_{t\in[0,T]}I_1(t)\leq\!\!\!\!\!\!\!\!&&C_{T}(1+|x|^3+|y|^3)(\frac{\ep}{\delta^{1/2}}+\ep+\frac{\ep^2}{\delta}+\delta^2),\label{I_1}
\end{eqnarray}
which will be proved in Step 2, we obtain
\begin{eqnarray*}
\sup_{t\in[0,T]}\EE|\hat{X}_{t}^{\ep}-\bar{X}_{t}|^2\leq\!\!\!\!\!\!\!\!&&C_{T}(1+|x|^3+|y|^3)\left(\frac{\ep}{\delta^{1/2}}+\ep+\frac{\ep^2}{\delta}+\delta\right)+\int^{T}_0\EE|\hat X^{\ep}_t-\bar{X}_{t}|^2dt.
\end{eqnarray*}
Hence, the Grownall's inequality yields
\begin{eqnarray*}
\sup_{t\in[0,T]}\EE|\hat{X}_{t}^{\ep}-\bar{X}_{t}|^2\leq\!\!\!\!\!\!\!\!&&C_{T}(1+|x|^3+|y|^3)\left(\frac{\ep}{\delta^{1/2}}+\ep+\frac{\ep^2}{\delta}+\delta\right),
\end{eqnarray*}
which completes the proof.

\vspace{0.3cm}
\textbf{Step 2.} In this step, we intend to prove estimate \eref{I_1}. Note that
\begin{eqnarray}    \label{J2}
&&\left|\int_{0}^{t}\left[b(s(\delta), X_{s(\delta)}^{\ep},\mathscr{L}_{X^{\ep}_{s(\delta)}},\hat{Y}_{s}^{\ep})-\bar{b}(s(\delta), X^{\ep}_{s(\delta)},\mathscr{L}_{X^{\ep}_{s(\delta)}})\right]ds\right|^2\nonumber\\
\leq\!\!\!\!\!\!\!\!&&2\left|\sum_{k=0}^{[t/\delta]-1}
\int_{k\delta}^{(k+1)\delta}\left[b(k\delta, X_{k\delta}^{\ep},\mathscr{L}_{X^{\ep}_{k\delta}},\hat{Y}_{s}^{\ep})-\bar{b}(k\delta, X_{k\delta}^{\ep},\mathscr{L}_{X^{\ep}_{k\delta}})\right]ds\right|^2\nonumber\\
&&+2\left|\int_{t(\delta)}^{t}\left[b(t(\delta), X_{t(\delta)}^{\ep},\mathscr{L}_{X^{\ep}_{t(\delta)}}\hat{Y}_{s}^{\ep})-\bar{b}(t(\delta), X_{t(\delta)}^{\ep},\mathscr{L}_{X^{\ep}_{t(\delta)}})\right]ds\right|^2\nonumber\\
=\!\!\!\!\!\!\!\!&&2\sum_{k=0}^{[t/\delta]-1}
\left|\int_{k\delta}^{(k+1)\delta}\left[b(k\delta, X_{k\delta}^{\ep},\mathscr{L}_{X^{\ep}_{k\delta}},\hat{Y}_{s}^{\ep})-\bar{b}(k\delta, X_{k\delta}^{\ep},\mathscr{L}_{X^{\ep}_{k\delta}})\right]ds\right|^2\nonumber\\
&&+4\!\!\!\!\!\sum_{0\leq i<j\leq [t/\delta]-1}\left\langle \int_{i\delta}^{(i+1)\delta}\left[b(i\delta, X_{i\delta}^{\ep},\mathscr{L}_{X^{\ep}_{i\delta}},\hat{Y}_{s}^{\ep})-\bar{b}(i\delta, X_{i\delta}^{\ep},\mathscr{L}_{X^{\ep}_{i\delta}})\right]ds, \right.\nonumber\\
&&\quad\quad\quad\quad\quad\quad\left.\int_{j\delta}^{(j+1)\delta}\left[b(j\delta, X_{j\delta}^{\ep},\mathscr{L}_{X^{\ep}_{j\delta}},\hat{Y}_{s}^{\ep})-\bar{b}(j\delta, X_{j\delta}^{\ep},\mathscr{L}_{X^{\ep}_{j\delta}})\right]ds\right\rangle\nonumber\\
&&+2\left|\int_{t(\delta)}^{t}\left[b(t(\delta),X_{t(\delta)}^{\ep},\mathscr{L}_{X^{\ep}_{t(\delta)}},\hat{Y}_{s}^{\ep})-\bar{b}(t(\delta),X_{t(\delta)}^{\ep},\mathscr{L}_{X^{\ep}_{t(\delta)}})\right]ds\right|^2\nonumber\\
:=\!\!\!\!\!\!\!\!&&\sum^3_{i=1}I_{1i}(t).\label{I1F}
\end{eqnarray}

For $I_{13}(t)$, by estimate \eref{4.17} below, Lemmas \ref{PMY} and \ref{MDY}, it is easy to prove that
\begin{eqnarray}
\sup_{t\in [0,T]}\EE I_{13}(t)\leq\!\!\!\!\!\!\!\!&&C_T\delta\EE\int_{t(\delta)}^{t}\left[1+|X^{\ep}_{t(\delta)}|^{2}+\EE|X^{\ep}_{t(\delta)}|^{2}+|\hat{Y}_{s}^{\ep}|^{2}\right]ds\nonumber\\
\leq\!\!\!\!\!\!\!\!&& C_{T}(1+|x|^2+|y|^2)\delta^2.\label{I13}
\end{eqnarray}

For the term $I_{11}(t)$, we have
\begin{eqnarray*}
\mathbb{E}I_{11}(t)=\!\!\!\!\!\!\!\!&&2\sum_{k=0}^{[t/\delta]-1}\mathbb{E}\left|\int_{k\delta}^{(k+1)\delta}\left[b(k\delta, X_{k\delta}^{\ep},\mathscr{L}_{X^{\ep}_{k\delta}},\hat{Y}_{s}^{\ep})-\bar{b}(k\delta,X_{k\delta}^{\ep},\mathscr{L}_{X^{\ep}_{k\delta}})\right]ds\right|^{2}\nonumber\\
=\!\!\!\!\!\!\!\!&&2\sum_{k=0}^{[t/\delta]-1}\mathbb{E}\left|\int_{0}^{\frac{\delta}{\ep}}
\left[b(k\delta, X_{k\delta}^{\ep},\mathscr{L}_{X^{\ep}_{k\delta}},\hat{Y}_{s\ep+k\delta}^{\ep})-\bar{b}(k\delta, X_{k\delta}^{\ep},\mathscr{L}_{X^{\ep}_{k\delta}})\right]ds\right|^{2}  \nonumber\\
=\!\!\!\!\!\!\!\!&&2\ep^2\sum_{k=0}^{[t/\delta]-1}\int_{0}^{\frac{\delta}{\ep}}
\int_{r}^{\frac{\delta}{\ep}}\Psi_{k}(s,r)dsdr,  \nonumber
\end{eqnarray*}
where for any $0\leq r\leq s\leq \frac{\delta}{\ep}$,
\begin{eqnarray*}
\Psi_{k}(s,r):=\!\!\!\!\!\!\!\!&&\mathbb{E}\left[
\big\langle b(k\delta, X_{k\delta}^{\ep},\mathscr{L}_{X^{\ep}_{k\delta}},\hat{Y}_{s\ep+k\delta}^{\ep})-\bar{b}(k\delta, X_{k\delta}^{\ep},\mathscr{L}_{X^{\ep}_{k\delta}}),\right.\\
&&\quad\quad\left.b(k\delta, X_{k\delta}^{\ep},\mathscr{L}_{X^{\ep}_{k\delta}},\hat{Y}_{r\ep+k\delta}^{\ep})-\bar{b}(k\delta, X_{k\delta}^{\ep},\mathscr{L}_{X^{\ep}_{k\delta}})\big\rangle\right].
\end{eqnarray*}
For any $s>0$, $\mu\in\mathscr{P}_2$ and random variables $x,y\in\mathscr{F}_s$, we consider the following equation
\begin{eqnarray*}
\tilde{Y}^{\ep,s,x,\mu,y}_t=y+\frac{1}{\ep}\int^t_s f(s,x,\mu,\tilde{Y}^{\ep,s,x,\mu,y}_r)dr+\frac{1}{\sqrt{\ep}}\int^t_s g(s,x,\mu,\tilde{Y}^{\ep,s,x,\mu,y}_r)dW^2_r,\quad t\geq s.
\end{eqnarray*}
Then by the construction of $\hat{Y}_{t}^{\ep}$,
for any $k\in \mathbb{N}_{\ast}$, we have
$$
\hat{Y}_{t}^{\ep}=\tilde Y^{\ep,k\delta,X_{k\delta}^{\ep},\mathscr{L}_{X^{\ep}_{k\delta}},\hat{Y}_{k\delta}^{\ep}}_t,\quad t\in[k\delta,(k+1)\delta],
$$
which implies
\begin{eqnarray*}
\Psi_{k}(s,r)=\!\!\!\!\!\!\!\!&&\mathbb{E}\left[
\langle b(k\delta, X_{k\delta}^{\ep},\mathscr{L}_{X^{\ep}_{k\delta}},\tilde {Y}^{\ep, k\delta, X_{k\delta}^{\ep}, \mathscr{L}_{X^{\ep}_{k\delta}},\hat Y_{k\delta}^{\ep}}_{s\ep+k\delta})-\bar{b}(k\delta, X_{k\delta}^{\ep},\mathscr{L}_{X^{\ep}_{k\delta}}),\right.\\
&&\quad\quad\quad \left.b(k\delta, X_{k\delta}^{\ep},\mathscr{L}_{X^{\ep}_{k\delta}},\tilde{Y}^{\ep, k\delta,X_{k\delta}^{\ep}, \mathscr{L}_{X^{\ep}_{k\delta}},\hat Y_{k\delta}^{\ep}}_{r\ep+k\delta})-\bar{b}(k\delta, X_{k\delta}^{\ep},\mathscr{L}_{X^{\ep}_{k\delta}})\rangle \right].
\end{eqnarray*}
Note that since for any fixed $x\in\RR^n$, $y\in\RR^m$, $\tilde Y^{\ep, k\delta, x, \mu, y}_{s\ep+k\delta}$ is independent of $\mathscr{F}_{k\delta}$, and $X_{k\delta}^{\ep}$ , $\hat Y_{k\delta}^{\ep}$ are $\mathscr{F}_{k\delta}$-measurable, we have
\begin{eqnarray*}
\Psi_{k}(s,r)=\!\!\!\!\!\!\!\!&&\EE\Bigg\{\mathbb{E}\left[
\langle b(k\delta, X_{k\delta}^{\ep},\mathscr{L}_{X^{\ep}_{k\delta}},\tilde {Y}^{\ep, k\delta, X_{k\delta}^{\ep},\mathscr{L}_{X^{\ep}_{k\delta}}, \hat Y_{k\delta}^{\ep}}_{s\ep+k\delta})-\bar{b}(k\delta, X_{k\delta}^{\ep},\mathscr{L}_{X^{\ep}_{k\delta}}),\right.\\
&&\quad\quad\quad \left.b(k\delta, X_{k\delta}^{\ep},\mathscr{L}_{X^{\ep}_{k\delta}},\tilde{Y}^{\ep, k\delta, X_{k\delta}^{\ep}, \mathscr{L}_{X^{\ep}_{k\delta}},\hat Y_{k\delta}^{\ep}}_{r\ep+k\delta})-\bar{b}(k\delta, X_{k\delta}^{\ep},\mathscr{L}_{X^{\ep}_{k\delta}})\rangle \Big|\mathscr{F}_{k\delta}\right](\omega)\Bigg\}\\
=\!\!\!\!\!\!\!\!&&\EE\Bigg\{\EE\left[\langle b(k\delta, X^{\ep}_{k\delta}(\omega),\mathscr{L}_{X^{\ep}_{k\delta}},\tilde {Y}^{\ep, k\delta, X^{\ep}_{k\delta}(\omega), \mathscr{L}_{X^{\ep}_{k\delta}},\hat Y_{k\delta}^{\ep}(\omega)}_{s\ep+k\delta})-\bar{b}(k\delta, X^{\ep}_{k\delta}(\omega),\mathscr{L}_{X^{\ep}_{k\delta}}),\right.\\
&&\quad\quad\quad\left.b(k\delta, X^{\ep}_{k\delta}(\omega),\mathscr{L}_{X^{\ep}_{k\delta}},\tilde {Y}^{\ep, k\delta, X^{\ep}_{k\delta}(\omega),\mathscr{L}_{X^{\ep}_{k\delta}}, \hat Y_{k\delta}^{\ep}(\omega)}_{r\ep+k\delta})-\bar{b}(k\delta, X^{\ep}_{k\delta}(\omega),\mathscr{L}_{X^{\ep}_{k\delta}})\rangle\right] \Bigg\}.
\end{eqnarray*}
By the definition of the process $\{\tilde{Y}^{\ep,s,x,\mu,y}_t\}_{t\geq 0}$, it is easy to see that
\begin{eqnarray}
\tilde{Y}^{\ep,k\delta,x,\mu,y}_{s\ep+k\delta}=\!\!\!\!\!\!\!\!&&y+\frac{1}{\ep}\int^{s\ep+k\delta}_{k\delta} f(k\delta,x,\mu,\tilde{Y}^{\ep,k\delta,x,\mu,y}_r)dr+\frac{1}{\sqrt{\ep}}\int^{s\ep+k\delta}_{k\delta} g(k\delta,x,\mu,\tilde{Y}^{\ep,k\delta,x,\mu,y}_r)dW^2_r\nonumber\\
=\!\!\!\!\!\!\!\!&&y+\frac{1}{\ep}\int^{s\ep}_{0} f(k\delta,x,\mu,\tilde{Y}^{\ep,k\delta,x,\mu,y}_{r+k\delta})dr+\frac{1}{\sqrt{\ep}}\int^{s\ep}_{0} g(k\delta,x,\mu,\tilde{Y}^{\ep,k\delta,x,\mu,y}_{r+k\delta})dW^{2,k\delta}_r\nonumber\\
=\!\!\!\!\!\!\!\!&&y+\int^{s}_{0} f(k\delta,x,\mu,\tilde{Y}^{\ep,k\delta,x,\mu,y}_{r\ep+k\delta})dr+\int^{s}_{0} g(k\delta,x,\mu,\tilde{Y}^{\ep,k\delta,x,\mu,y}_{r\ep+k\delta})d\hat{W}^{2,k\delta}_r,\label{E3.15}
\end{eqnarray}
where $\{W^{2, k\delta}_r:=W^2_{r+k\delta}-W^2_{k\delta}\}_{r\geq 0}$ and $\{\hat W^{2,k\delta}_t:=\frac{1}{\sqrt{\ep}}W^{2,k\delta}_{t\ep}\}_{t\geq 0}$. Recall the solution of the frozen equation satisfies
\begin{eqnarray}
Y_{s}^{k\delta, x, \mu,y}=\!\!\!\!\!\!\!\!&& y
+\int_{0}^{{s}}f(k\delta, x, \mu,Y_{r}^{k\delta, x,\mu,y})dr
+\int_{0}^{{s}}g(k\delta, x, \mu,Y_{r}^{k\delta, x,\mu,y})d\tilde{W}^2_r.  \label{E3.16}
\end{eqnarray}
The uniqueness of the solutions of Eq. (\ref{E3.15}) and Eq. (\ref{E3.16}) implies
that the distribution of $\{\tilde Y^{\ep, k\delta, x,\mu,y}_{s\ep+k\delta}\}_{0\leq s\leq \delta/\ep}$
coincides with the distribution of $\{Y_{s}^{k\delta,x, \mu,y}\}_{0\leq s\leq \delta/\ep}$.
Then by Proposition \ref{Rem 4.1}, we have
\begin{eqnarray*}
\Psi_{k}(s,r)=\!\!\!\!\!\!\!\!&&\EE\Big[\tilde \EE\big\langle b(k\delta, X^{\ep}_{k\delta}(\omega),\mathscr{L}_{X^{\ep}_{k\delta}},Y^{k\delta, X^{\ep}_{k\delta}(\omega),\mathscr{L}_{X^{\ep}_{k\delta}},\hat Y_{k\delta}^{\ep}(\omega)}_{s})-\bar{b}(k\delta,  X^{\ep}_{k\delta}(\omega),\mathscr{L}_{X^{\ep}_{k\delta}}),\\
&&\quad\quad b(k\delta,  X^{\ep}_{k\delta}(\omega),\mathscr{L}_{X^{\ep}_{k\delta}},Y^{k\delta,  X^{\ep}_{k\delta}(\omega),\mathscr{L}_{X^{\ep}_{k\delta}},\hat Y_{k\delta}^{\ep}(\omega)}_{r})-\bar{b}(k\delta,  X^{\ep}_{k\delta}(\omega),\mathscr{L}_{X^{\ep}_{k\delta}})\big\rangle\Big]\\
=\!\!\!\!\!\!\!\!&&\EE\Big[\tilde \EE\big\langle \tilde \EE\big[ b(k\delta, X^{\ep}_{k\delta}(\omega),\mathscr{L}_{X^{\ep}_{k\delta}},Y^{k\delta, X^{\ep}_{k\delta}(\omega),\mathscr{L}_{X^{\ep}_{k\delta}},\hat Y_{k\delta}^{\ep}(\omega)}_{s})|\tilde{\mathscr{F}}_r\big](\tilde \omega)-\bar{b}(k\delta,  X^{\ep}_{k\delta}(\omega),\mathscr{L}_{X^{\ep}_{k\delta}}),\\
&&\quad\quad b(k\delta,  X^{\ep}_{k\delta}(\omega),\mathscr{L}_{X^{\ep}_{k\delta}},Y^{k\delta,  X^{\ep}_{k\delta}(\omega),\mathscr{L}_{X^{\ep}_{k\delta}},\hat Y_{k\delta}^{\ep}(\omega)}_{r}(\tilde \omega))-\bar{b}(k\delta,  X^{\ep}_{k\delta}(\omega),\mathscr{L}_{X^{\ep}_{k\delta}})\big\rangle\Big]\\
\leq\!\!\!\!\!\!\!\!&&C_T\EE\left\{\tilde \EE\left[1+|X^{\ep}_{k\delta}(\omega)|^{2}+|Y_{r}^{k\delta, X^{\ep}_{k\delta}(\omega), \mathscr{L}_{X^{\ep}_{k\delta}},\hat Y_{k\delta}^{\ep}(\omega)}(\tilde \omega)|^{2}+\mathscr{L}_{X^{\ep}_{k\delta}}(|\cdot|^2)\right] e^{-\frac{(s-r)\beta}{2}}\right\}\\
\leq\!\!\!\!\!\!\!\!&&C_T\EE\left(1+|X^{\ep}_{k\delta}|^{2}+|\hat Y_{k\delta}^{\ep}|^{2}+\EE|X^{\ep}_{k\delta}|^{2}\right)e^{-\frac{(s-r)\beta}{2}}\\
\leq\!\!\!\!\!\!\!\!&&C_{T}(1+|x|^2+|y|^2)e^{-\frac{(s-r)\beta}{2}},
\end{eqnarray*}
where the last inequality is consequence of Lemmas \ref{PMY} and \ref{MDY}. Hence we have
\begin{eqnarray}
\sup_{t\in[0,T]}\mathbb{E}I_{11}(t)\leq\!\!\!\!\!\!\!\!&&C_{T}(1+|x|^2+|y|^2)\frac{\ep^2}{\delta}
\int_{0}^{\frac{\delta}{\ep}}\int_{r}^{\frac{\delta}{\ep}}e^{-\frac{(s-r)\beta}{2}}dsdr\nonumber\\
=\!\!\!\!\!\!\!\!&&C_{T}(1+|x|^2+|y|^2)\frac{\ep^2}{\delta}\Big(\frac{\delta}{\beta\ep}-\frac{1}{\beta^{2}}
+\frac{1}{\beta^{2}}e^{-\frac{\beta\delta}{\ep}}\Big)  \nonumber\\
\leq\!\!\!\!\!\!\!\!&&C_{T}(1+|x|^2+|y|^2)(\ep+\frac{\ep^2}{\delta}).\label{I11}
\end{eqnarray}

For the term $I_{12}(t)$, in Step 3 we will prove the following estimate:
\begin{eqnarray}
\sup_{t\in[0,T]}\mathbb{E}I_{12}(t)\leq\!\!\!\!\!\!\!\!&&C_{T}(1+|x|^3+|y|^3)(\frac{\ep}{\delta^{1/2}}+\ep).\label{I12}
\end{eqnarray}

As a consequence, estimates \eref{I1F}, \eref{I13}, \eref{I11} and \eref{I12} imply \eref{I_1}.

\vspace{0.3cm}
\textbf{Step 3.} In this step, we intend to prove estimate \eref{I12}. For convenience, for any $i\in \mathbb{N}$, setting $Z^{\ep}_{i,t}:=\tilde{Y}^{\ep,i\delta,X^{\ep}_{i\delta},\mathscr{L}_{X^{\ep}_{i\delta}},\hat{Y}^{\ep}_{i\delta}}_t$ with $i\delta\leq t$, we obtain that
\begin{equation}\left\{\begin{array}{l}\label{AP}
\displaystyle
dZ^{\ep}_{i,t}=\frac{1}{\ep}f(i\delta, X^{\ep}_{i\delta}, \mathscr{L}_{X^{\ep}_{i\delta}}Z^{\ep}_{i,t})dt+\frac{1}{\sqrt{\ep}}g(i\delta,X^{\ep}_{i\delta},\mathscr{L}_{X^{\ep}_{i\delta}}, Z^{\ep}_{i,t})dW_{t}^{2},\\
Z^{\ep}_{i,i\delta}=\hat{Y}^{\ep}_{i\delta}.
\end{array}\right.
\end{equation}
By the definition above, it is easy to see that
\begin{eqnarray*}
Z^{\ep}_{k,t}=\hat{Y}^{\ep}_t,\quad t\in [k\delta, (k+1)\delta]
\end{eqnarray*}
and continuity implies that
\begin{eqnarray*}
Z^{\ep}_{k,(k+1)\delta}=Z^{\ep}_{k+1,(k+1)\delta}=\hat{Y}^{\ep}_{(k+1)\delta}.
\end{eqnarray*}
Let $\EE_{s}$ be the conditional expectation $w.r.t.$ $\mathscr{F}_{s}$,$s\geq 0$.  Then for any $0\leq i<j\leq [t/\delta]-1$,
\begin{eqnarray}
&&\EE\left\langle \int_{i\delta}^{(i+1)\delta}\left[b(i\delta, X_{i\delta}^{\ep},\mathscr{L}_{X^{\ep}_{i\delta}}, \hat{Y}_{s}^{\ep})-\bar{b}(i\delta, X_{i\delta}^{\ep},\mathscr{L}_{X^{\ep}_{i\delta}})\right]ds, \right.\nonumber\\
&&\quad\quad\left.\int_{j\delta}^{(j+1)\delta}\left[b(j\delta, X_{j\delta}^{\ep},\mathscr{L}_{X^{\ep}_{j\delta}}, \hat{Y}_{s}^{\ep})-\bar{b}(j\delta, X_{j\delta}^{\ep},\mathscr{L}_{X^{\ep}_{j\delta}})\right]ds\right\rangle\nonumber\\
=\!\!\!\!\!\!\!\!&&\int_{i\delta}^{(i+1)\delta}\int_{j\delta}^{(j+1)\delta}\EE\left\langle b(i\delta, X_{i\delta}^{\ep},\mathscr{L}_{X^{\ep}_{i\delta}},\hat{Y}_{s}^{\ep})-\bar{b}(i\delta, X_{i\delta}^{\ep},\mathscr{L}_{X^{\ep}_{i\delta}}), \right.\nonumber\\
&&\quad\quad\quad\quad\quad\quad\quad\quad\left. b(j\delta, X_{j\delta}^{\ep},\mathscr{L}_{X^{\ep}_{j\delta}},\hat{Y}_{t}^{\ep})-\bar{b}(j\delta, X_{j\delta}^{\ep},\mathscr{L}_{X^{\ep}_{j\delta}})\right\rangle dsdt\nonumber\\
\leq\!\!\!\!\!\!\!\!&&\int_{i\delta}^{(i+1)\delta}\int_{j\delta}^{(j+1)\delta}\EE\left\{\left |b(i\delta, X_{i\delta}^{\ep},\mathscr{L}_{X^{\ep}_{i\delta}},\hat{Y}_{s}^{\ep})-\bar{b}(i\delta, X_{i\delta}^{\ep},\mathscr{L}_{X^{\ep}_{i\delta}})\right|\right.\nonumber\\
&&\quad\quad\quad\quad\quad\quad\quad\quad\left.\cdot\left|\EE_{(i+1)\delta}\left[b(j\delta, X_{j\delta}^{\ep},\mathscr{L}_{X^{\ep}_{j\delta}}\hat{Y}_{t}^{\ep})-\bar{b}(j\delta, X_{j\delta}^{\ep},\mathscr{L}_{X^{\ep}_{j\delta}})\right]\right| \right\}dsdt\nonumber\\
\leq\!\!\!\!\!\!\!\!&&C_T\int_{i\delta}^{(i+1)\delta}\!\!\int_{j\delta}^{(j+1)\delta}\!\!\EE\left\{(1+|X_{i\delta}^{\ep}|+|\hat{Y}_{s}^{\ep}|)\right.\big|\EE_{(i+1)\delta}\left[\left(b(j\delta, X_{j\delta}^{\ep},\mathscr{L}_{X^{\ep}_{j\delta}},\hat{Y}_{t}^{\ep})-\bar{b}(j\delta, X_{j\delta}^{\ep},\mathscr{L}_{X^{\ep}_{j\delta}})\right)\right.\nonumber\\
&&\quad\quad\quad\quad\left.\left.-\left(b((i+1)\delta, X_{(i+1)\delta}^{\ep}, \mathscr{L}_{X^{\ep}_{(i+1)\delta}},Z_{i+1, t}^{\ep})-\bar{b}((i+1)\delta, X_{(i+1)\delta}^{\ep},\mathscr{L}_{X^{\ep}_{(i+1)\delta}})\right)\right]\big|\right\}dsdt\nonumber\\
&&+C_T\int_{i\delta}^{(i+1)\delta}\int_{j\delta}^{(j+1)\delta}\EE\left\{(1+|X_{i\delta}^{\ep}|+|\hat{Y}_{s}^{\ep}|)\right.\nonumber\\
&&\left.\cdot\left|\EE_{(i+1)\delta}\left[b((i+1)\delta, X_{(i+1)\delta}^{\ep}, \mathscr{L}_{X^{\ep}_{(i+1)\delta}}, Z_{i+1, t}^{\ep})-\bar{b}((i+1)\delta, X_{(i+1)\delta}^{\ep},\mathscr{L}_{X^{\ep}_{(i+1)\delta}})\right]\right| \right\}dsdt\nonumber\\
:=\!\!\!\!\!\!\!\!&&B_1+B_2.\label{B1+B2}
\end{eqnarray}
On one hand, by a similar argument for $I_{11}(t)$, we obtain
\begin{eqnarray}
B_2\leq\!\!\!\!\!\!\!\!&&C_T\int_{i\delta}^{(i+1)\delta}\int_{j\delta}^{(j+1)\delta}\EE\left[(1+|X_{i\delta}^{\ep}|+|\hat{Y}_{s}^{\ep}|)(1+|X_{(i+1)\delta}^{\ep}|+|\hat{Y}_{(i+1)\delta}^{\ep}|)\right]e^{\frac{-\beta[t-(i+1)\delta]}{2\ep}}dsdt \nonumber\\
\leq\!\!\!\!\!\!\!\!&&C_{T}(1+|x|^2+|y|^2)\int_{i\delta}^{(i+1)\delta}\int_{j\delta}^{(j+1)\delta}e^{\frac{-\beta[t-(i+1)\delta]}{2\ep}}dsdt\nonumber\\
\leq\!\!\!\!\!\!\!\!&&C_{T}(1+|x|^2+|y|^2)\ep\delta e^{\frac{-\beta(j-i)\delta}{2\ep}}(1-e^{\frac{-\beta\delta}{2\ep}}).\label{B2}
\end{eqnarray}
On the other hand,
\begin{eqnarray*}
B_1=\!\!\!\!\!\!\!\!&&C_T\int_{i\delta}^{(i+1)\delta}\!\!\int_{j\delta}^{(j+1)\delta}\!\!\sum^{j-1}_{k=i+1}\!\!\EE\left\{(1+|X_{i\delta}^{\ep}|+|\hat{Y}_{s}^{\ep}|)\right.\big|\EE_{(i+1)\delta}\left[\left(b((k+1)\delta, X_{(k+1)\delta}^{\ep},\mathscr{L}_{X_{(k+1)\delta}^{\ep}}, Z_{k+1, t}^{\ep})\right.\right.\nonumber\\
&&\left.-\bar{b}((k+1)\delta, X_{(k+1)\delta}^{\ep},\mathscr{L}_{X_{(k+1)\delta}^{\ep}})\right)\left.\left.-\left(b(k\delta, X_{k\delta}^{\ep},\mathscr{L}_{X_{k\delta}^{\ep}},Z_{k, t}^{\ep})-\bar{b}(k\delta, X_{k\delta}^{\ep},\mathscr{L}_{X_{k\delta}^{\ep}})\right)\right]\big|\right\}dsdt\nonumber\\
=\!\!\!\!\!\!\!\!&&C_T\int_{i\delta}^{(i+1)\delta}\!\!\int_{j\delta}^{(j+1)\delta}\!\!\sum^{j-1}_{k=i+1}\!\!\EE\left\{(1+|X_{i\delta}^{\ep}|+|\hat{Y}_{s}^{\ep}|)\right.\big|\EE_{k\delta}\left[\left(b((k+1)\delta, X_{(k+1)\delta}^{\ep},\mathscr{L}_{X_{(k+1)\delta}^{\ep}}, Z_{k+1, t}^{\ep})\right.\right.\nonumber\\
&&\left.-\bar{b}((k+1)\delta, X_{(k+1)\delta}^{\ep},\mathscr{L}_{X_{(k+1)\delta}^{\ep}})\right)\left.\left.-\left(b(k\delta, X_{k\delta}^{\ep},\mathscr{L}_{X_{k\delta}^{\ep}},Z_{k, t}^{\ep})-\bar{b}(k\delta, X_{k\delta}^{\ep},\mathscr{L}_{X_{k\delta}^{\ep}})\right)\right]\big|\right\}dsdt.
\end{eqnarray*}
Thanks to the Markov property, we get
\begin{eqnarray*}
&&\EE_{k\delta}\left[b((k+1)\delta, X_{(k+1)\delta}^{\ep},\mathscr{L}_{X_{(k+1)\delta}^{\ep}},Z_{k+1, t}^{\ep})-\bar{b}((k+1)\delta, X_{(k+1)\delta}^{\ep},\mathscr{L}_{X_{(k+1)\delta}^{\ep}})\right]\\
=\!\!\!\!\!\!\!\!&&\EE_{k\delta}\left[\tilde b((k+1)\delta, X_{(k+1)\delta}^{\ep},\mathscr{L}_{X_{(k+1)\delta}^{\ep}},\hat Y_{(k+1)\delta}^{\ep}, [t-(k+1)\delta]/\ep)\right]
\end{eqnarray*}
and
\begin{eqnarray*}
&&\EE_{k\delta}\left[b(k\delta, X_{k\delta}^{\ep},\mathscr{L}_{X_{k\delta}^{\ep}},Z_{k, t}^{\ep})-\bar{b}(k\delta, X_{k\delta}^{\ep},\mathscr{L}_{X_{k\delta}^{\ep}})\right]\\
=\!\!\!\!\!\!\!\!&&\EE_{k\delta}\left[\tilde b(k\delta, X_{k\delta}^{\ep},\mathscr{L}_{X_{k\delta}^{\ep}},\hat Y_{(k+1)\delta}^{\ep},[t-(k+1)\delta]/\ep)\right],
\end{eqnarray*}
where  $\tilde b(t,x,\mu,y,s)=\tilde \EE b(t,x, \mu,Y^{t,x,\mu,y}_s)-\bar b(t,x,\mu)$.

Recall the following  properties of $\tilde b$ (see the detailed proof in Section 5.3):
\begin{itemize}
\item{For any $t_1,t_2\in[0,T], s\geq 0$, $x\in\RR^n$, $y\in\RR^m$ and $\mu\in\mathscr{P}_2$,
\begin{eqnarray}
&&|\tilde b(t_1, x, \mu, y, s)-\tilde b(t_2, x,\mu, y, s)|\nonumber\\
\leq\!\!\!\!\!\!\!\!&& C_T|t_1-t_2|e^{-\eta s}\left\{1+|x|^{\gamma_1}+|y|^{\gamma_1}+[\mu(|\cdot|^2)]^{\gamma_1/2}\right\}; \label{tilde bt}
\end{eqnarray}}
\item{For any $t\in[0, T], s\geq 0$, $x_1,x_2\in\RR^n$, $y\in\RR^m$ and $\mu\in\mathscr{P}_2$,
\begin{eqnarray}
&&|\tilde b(t, x_1, \mu, y, s)-\tilde b(t, x_2, \mu, y, s)|\nonumber\\
\leq\!\!\!\!\!\!\!\!&& C_T|x_1-x_2|e^{-\eta s}\left\{1+|x_1|^{\gamma_1}+|x_2|^{\gamma_1}+|y|^{\gamma_1}+[\mu(|\cdot|^2)]^{\gamma_1/2}\right\}; \label{tilde bx}
\end{eqnarray}
}
\item{For any $t\in[0,T], s\geq 0$, $x\in\RR^n$, $y\in\RR^m$ and $\mu_1,\mu_2\in\mathscr{P}_2$,
\begin{eqnarray}
&&|\tilde b(t, x, \mu_1, y, s)-\tilde b(t, x, \mu_2,y, s)|\nonumber\\
\leq\!\!\!\!\!\!\!\!&&C_T\mathbb{W}_2(\mu_1,\mu_2)e^{-\eta s}\left\{1+|x|^{\gamma_1}+|y|^{\gamma_1}+[\mu_1(|\cdot|^2)]^{\gamma_1/2}+[\mu_2(|\cdot|^2)]^{\gamma_1/2}\right\}, \label{tilde bmu}
\end{eqnarray}
}
\end{itemize}
where $\eta$ is a positive constant. Then by estimates \eref{tilde bt}-\eref{tilde bmu} and Lemma \ref{PMY}, we have
\begin{eqnarray}
B_1\leq\!\!\!\!\!\!\!\!&&C\int_{i\delta}^{(i+1)\delta}\!\!\int_{j\delta}^{(j+1)\delta}\!\sum^{j-1}_{k=i+1}\EE\left\{(1+|X_{i\delta}^{\ep}|+|\hat{Y}_{s}^{\ep}|)\right.\nonumber\\
&&\quad\quad\left.\left|\tilde b((k+1)\delta, X_{(k+1)\delta}^{\ep},\mathscr{L}_{X_{(k+1)\delta}^{\ep}},\hat Y_{(k+1)\delta}^{\ep}, [t-(k+1)\delta]/\ep)\right.\right.\nonumber\\
&&\quad\left.-\left.\tilde b(k\delta, X_{k\delta}^{\ep},\mathscr{L}_{X_{k\delta}^{\ep}},\hat Y_{(k+1)\delta}^{\ep},[t-(k+1)\delta]/\ep)\right|\right\}dsdt\nonumber\\
\leq\!\!\!\!\!\!\!\!&&C_T\int_{i\delta}^{(i+1)\delta}\!\!\int_{j\delta}^{(j+1)\delta}\!\sum^{j-1}_{k=i+1}\EE\left[(1+|X_{i\delta}^{\ep}|+|\hat{Y}_{s}^{\ep}|)(1+|X^{\ep}_{k\delta}|^{\gamma_1}+|X^{\ep}_{(k+1)\delta}|^{\gamma_1}+|\hat Y^{\ep}_{(k+1)\delta}|^{\gamma_1})\right.\nonumber\\
&&\quad\quad\quad\quad\left.\left(\delta+|X_{(k+1)\delta}^{\ep}-X_{k\delta}^{\ep}|+[\EE|X_{(k+1)\delta}^{\ep}-X_{k\delta}^{\ep}|^2]^{1/2}\right)\right] e^{\frac{-\beta [t-(k+1)\delta]}{4\ep}}dsdt\nonumber\\
\leq\!\!\!\!\!\!\!\!&&C_{T}(1+|x|^3+|y|^3)\delta^{1/2}\int_{i\delta}^{(i+1)\delta}\!\!\int_{j\delta}^{(j+1)\delta}\!\sum^{j-1}_{k=i+1}e^{\frac{-\beta [t-(k+1)\delta]}{4\ep}}dsdt\nonumber\\
\leq\!\!\!\!\!\!\!\!&&C_{T}(1+|x|^3+|y|^3)\delta^{1/2}\int_{i\delta}^{(i+1)\delta}\!\!\int_{j\delta}^{(j+1)\delta}\frac{e^{\frac{-\beta(t-j\delta)}{4\ep}}}{1-e^{\frac{-\beta\delta}{4\ep}}}dsdt\nonumber\\
\leq\!\!\!\!\!\!\!\!&&C_{T}(1+|x|^3+|y|^3)\delta^{3/2}\ep.\label{B1}
\end{eqnarray}
Combining estimates \eref{B1+B2}, \eref{B2} and \eref{B1}, we obtain
\begin{eqnarray*}
\sup_{t\in[0,T]}\mathbb{E}I_{12}(t)\leq\!\!\!\!\!\!\!\!&&C_{T}(1+|x|^3+|y|^3)\sum_{0\leq i<j\leq [T/\delta]-1}\left[\delta^{3/2}\ep+\ep\delta e^{\frac{-\beta(j-i)\delta}{4\ep}}(1-e^{\frac{-\beta\delta}{4\ep}})\right]\nonumber\\
\leq\!\!\!\!\!\!\!\!&&C_{T}(1+|x|^3+|y|^3)(\frac{\ep}{\delta^{1/2}}+\ep),
\end{eqnarray*}
which is the estimate \eref{I12}. The proof is complete.
\end{proof}

\vspace{0.2cm}
Now we are in a position to complete our first result.

\noindent\textbf{Proof of Theorem \ref{main result 1}:} Taking $\delta=\ep^{2/3}$, Lemmas $\ref{DEX}$ and $\ref{ESX}$ imply that for any $T>0$, initial values $x\in\RR^n$ and $y\in\RR^m$, there exists $C_T>0$ such that
\begin{eqnarray*}
\sup_{t\in[0, T]}\mathbb{E}|X_{t}^{\ep}-\bar{X}_{t}|^2\leq\!\!\!\!\!\!\!\!&& C_{T}(1+|x|^3+|y|^3)\ep^{2/3}.
\end{eqnarray*}
which proves the first part of Theorem \ref{main result 1}, i.e., $\eref{R1}$ holds.

Furthermore, if there is no noise in the slow equation (i.e., $\sigma=0$), we can improve the H\"{o}lder continuity in time in Lemma \ref{COX}, i.e., for any $T>0$, $0\leq t\leq t+h\leq T$, there exists a positive constant $C_{T}$ such that
\begin{eqnarray*}
\sup_{\ep\in(0,1)}\mathbb{E}|X_{t+h}^{\ep}-X_{t}^{\ep}|^{2}\leq C_{T}(1+|x|^2+|y|^2)h^2.
\end{eqnarray*}
Then, following almost the same procedure as above,  it is easy to see that
\begin{eqnarray*}
\sup_{t\in[0, T]}\mathbb{E}|X_{t}^{\ep}-\bar{X}_{t}|^2\leq\!\!\!\!\!\!\!\!&& C_{T}(1+|x|^3+|y|^3)\left(\ep+\frac{\ep^2}{\delta}+\delta^2\right).
\end{eqnarray*}
Hence,  taking $\delta=\ep$ yields $\eref{R2}$. The proof is complete.

\section{Proof of Theorem \ref{main result 2}}

In this section, we will use the technique of Poisson equation to prove the strong convergence order, which is quite different from the method used in Section 3. Because we will study the regularity of second-order derivatives of the solution for the corresponding Poisson equation, more conditions (see assumption \ref{A3}) are needed. This section is divided into two subsections. In Subsection 4.1, we study the regularity of the solution for the corresponding Poisson equation. In Subsection 4.2, we prove Theorem \ref{main result 2} by using the technique of Poisson equation. Note that we always assume conditions \ref{A1}-\ref{A3} hold.

\vskip 0.2cm
\subsection{Poisson equation}
Consider the following Poisson equation:
\begin{equation}
-\mathscr{L}_{2}(t,x,\mu)\Phi(t,x,\mu,y)=b(t,x,\mu,y)-\bar{b}(t,x,\mu),\label{PE}
\end{equation}
where
\begin{eqnarray*}
&&\Phi(t,x,\mu,y)=(\Phi_1(t,x,\mu,y),\ldots, \Phi_n(t,x,\mu,y));\\
&&\mathscr{L}_{2}(t,x,\mu)\Phi(t,x,\mu,y):=(\mathscr{L}_{2}(t,x,\mu)\Phi_1(t,x,\mu,y),\ldots, \mathscr{L}_{2}(t,x,\mu)\Phi_n(t,x,\mu,y))
\end{eqnarray*}
and for any $k=1,\ldots,n.$
\begin{eqnarray*}
\mathscr{L}_{2}(t,x,\mu)\Phi_k(t,x,\mu,y):=\!\!\!\!\!\!\!\!&&\langle f(t,x,\mu,y), \partial_y \Phi_k(t,x,\mu,y)\rangle\\
&&+\frac{1}{2}\text{Tr}[g g^{*}(t,x,\mu,y)\partial^2_{yy} \Phi_k(t,x,\mu,y)].
\end{eqnarray*}

The smoothness of the solution of the Poisson equation with respect to parameters have been studied in many references, see \cite{PV1,PV2,RSX} for example. Note that here the solution for the Poisson equation \eref{PE} depends on the parameter $\mu$, so here we have to check the regularity $w.r.t.$ $\mu$. The main result of this subsection is the following:
\begin{proposition}\label{P3.6}
Assume the assumptions \ref{A1}-\ref{A3} hold. Define
\begin{eqnarray}
\Phi(t,x,\mu,y):=\int^{\infty}_{0}\tilde \EE[b(t,x,\mu,Y^{t,x,\mu,y}_s)]-\bar{b}(t,x,\mu)ds.\label{SPE}
\end{eqnarray}
Then $\Phi(t,x,\mu,y)$ is the unique solution of Eq. \eref{PE} and it satisfies that $\Phi(\cdot,\cdot,\mu,\cdot)\in C^{1,2,2}([0,\infty)\times\RR^n\times\RR^m,\RR^n)$, $\Phi(t,x,\cdot,y)\in C^{1,1}(\mathscr{P}_2, \RR^n)$. Moreover, for any $t\in[0, T]$,
\begin{eqnarray}
&&\max\{|\Phi(t,x,\mu,y)|,\|\partial_y \Phi(t,x,\mu,y)\|,|\partial_t \Phi(t,x,\mu,y)|,\|\partial_x \Phi(t,x,\mu,y)\|, \|\partial_{\mu}\Phi(t,x,\mu,y)\|_{L^2(\mu)}\}\nonumber\\
\leq\!\!\!\!\!\!\!\!&& C_T\{1+|x|+|y|+[\mu(|\cdot|^2)]^{1/2}\} \label{E1}
\end{eqnarray}
and
\begin{eqnarray}
&&\max\{\|\partial^2_{xx} \Phi(t,x,\mu,y)\| ,\|\partial_{z}\partial_{\mu}\Phi(t,x,\mu,y)(\cdot)\|_{L^2(\mu)}\}\nonumber\\
\leq\!\!\!\!\!\!\!\!&& C_T\{1+|x|+|y|+[\mu(|\cdot|^2)]^{1/2}\}.\label{E2}
\end{eqnarray}
\end{proposition}
\begin{proof}
We will divide the proof into three steps.

\vspace{0.2cm}
\textbf{Step 1.}
Noting that $\mathscr{L}_{2}(t,x,\mu)$ is the infinitesimal generator of the frozen process $\{Y^{t,x,\mu}_s\}$, we easily check  that \eref{SPE} is the unique solution of the Poisson equation \eref{PE} under the assumptions \ref{A1}-\ref{A3}. Moreover, by a straightforward computation, we also have that $\Phi(\cdot,\cdot,\mu,\cdot)\in C^{1,2,2}([0,\infty)\times\RR^n\times\RR^m,\RR^n)$, $\Phi(t,x,\cdot,y)\in C^{1,1}(\mathscr{P}_2, \RR^n)$.

By Proposition \ref{Rem 4.1}, we get
\begin{eqnarray*}
|\Phi(t,x,\mu,y)|\leq\!\!\!\!\!\!\!\!&&\int^{\infty}_{0}|\tilde \EE[b(t,x,\mu,Y^{t,x,\mu,y}_s)]-\bar{b}(t,x,\mu)|ds\\
\leq\!\!\!\!\!\!\!\!&& C_T\{1+|x|+|y|+[\mu(|\cdot|^2)]^{1/2}\}\int^{\infty}_{0}e^{-\frac{\beta s}{2}}ds\\
\leq\!\!\!\!\!\!\!\!&& C_T\{1+|x|+|y|+[\mu(|\cdot|^2)]^{1/2}\}.
\end{eqnarray*}
By Lemma \ref{L3.6}, we have $\tilde \EE\|\partial_{y} Y^{t,x,\mu,y}_s\|^2\leq C_T e^{-\beta s}$, which implies
$$
\|\partial_y \Phi(t,x,\mu,y)\|\leq C_T.
$$
Furthermore, the remaining estimates in \eref{E1} can be obtained easily by \eref{tilde bt}-\eref{tilde bmu}. Therefore, it is sufficient to estimate \eref{E2} below.

We first recall that (see Subsection \ref{sub6.3} in the Appendix)
\begin{eqnarray*}
\tilde b_{s_0}(t, x, \mu, y, s)=\hat b(t, x, \mu, y, s)-\hat b(t, x, \mu, y, s+s_0),
\end{eqnarray*}
where $\hat b(t, x, \mu, y, s)=\tilde \EE b(t,x, \mu, Y^{t,x,\mu, y}_s)$. Note that
$$
\lim_{s_0\rightarrow \infty} \tilde b_{s_0}(t, x, \mu, y, s)=\tilde \EE[b(t,x,\mu,Y^{t,x,\mu,y}_s)]-\bar{b}(t,x,\mu).
$$
So, in order to prove \eref{E2},  it suffices to show there exists $\eta>0$ such that for any $s_0>0$, $t\in[0,T], s\geq 0$, $x\in\RR^n$, $y\in\RR^m$ and $\mu\in\mathscr{P}_2$,
\begin{eqnarray}
\|\partial^2_{xx} \tilde b_{s_0}(t,x,\mu,y,s)\|\leq C_Te^{-\eta s}\{1+|x|+|y|+[\mu(|\cdot|^2)]^{1/2}\}\label{E21}
\end{eqnarray}
and
\begin{eqnarray}
\|\partial_{z}\partial_{\mu}\Phi(t,x,\mu,y)(\cdot)\|_{L^2(\mu)}\leq C_T e^{-\eta s}\{1+|x|+|y|+[\mu(|\cdot|^2)]^{1/2}\},\label{E22}
\end{eqnarray}
which will be proved in the following two steps.
\vspace{0.3cm}

\textbf{Step 2.} In this step, we intend to prove estimate \eref{E21}. We recall that in \eref{MP} below
\begin{eqnarray*}
\tilde b_{s_0}(t, x, \mu, y, s)=\!\!\!\!\!\!\!\!&& \hat b(t, x, \mu, y, s)-\tilde \EE \hat b(t, x, \mu, Y^{t,x,\mu,y}_{s_0},s).
\end{eqnarray*}
Then the chain rule yields
\begin{eqnarray*}
\partial_{x}\tilde b_{s_0}(t, x, \mu, y, s)=\!\!\!\!\!\!\!\!&& \partial_{x} \hat b(t, x, \mu, y, s)-\tilde \EE \partial_{x}\hat b(t, x, \mu, Y^{t,x,\mu,y}_{s_0},s)\nonumber\\
&&-\tilde \EE \left[\partial_y\hat b(t, x, \mu, Y^{t,x,\mu, y}_{s_0},s) \cdot \partial_{x} Y^{t,x,\mu,y}_{s_0}\right],
\end{eqnarray*}
and furthermore,
\begin{eqnarray*}
\partial^2_{xx}\tilde b_{s_0}(t, x, \mu, y, s)=\!\!\!\!\!\!\!\!&& \partial^2_{xx} \hat b(t, x, \mu, y, s)-\tilde \EE \partial^2_{xx}\hat b(t, x, \mu, Y^{t,x,\mu,y}_{s_0},s)\nonumber\\
&&-\tilde \EE \left[\partial^2_{xy}\hat b(t, x, \mu, Y^{t,x,\mu, y}_{s_0},s)\cdot \partial_{x} Y^{t,x,\mu,y}_{s_0}\right]\\
&&-\tilde \EE \left\{\left[\partial^2_{yx}\hat b(t, x, \mu, Y^{t,x,\mu, y}_{s_0},s)+\partial^2_{yy}\hat b(t, x, \mu, Y^{t,x,\mu, y}_{s_0},s) \cdot \partial_{x} Y^{t,x,\mu,y}_{s_0}\right]\cdot \partial_{x} Y^{t,x,\mu,y}_{s_0}\right\}\\
&&-\tilde \EE \left[ \partial_{y}\hat b(t, x, \mu, Y^{t,x,\mu, y}_{s_0},s) \cdot \partial^2_{xx} Y^{t,x,\mu,y}_{s_0} \right]\\
:=\!\!\!\!\!\!\!\!&&\sum^4_{i=1} J_i.
\end{eqnarray*}

(i) For the term $J_1$, note that
\begin{eqnarray*}
\partial_{x}\hat b(t, x, \mu, y, s)=\!\!\!\!\!\!\!\!&&\tilde \EE \left[\partial_x b(t, x, \mu, Y^{t,x,\mu, y}_{s})\right]+\tilde\EE\left[ \partial_{y}b(t,x,\mu,Y^{t,x,\mu,y}_{s})\cdot\partial_{x}Y^{t,x,\mu,y}_s \right],
\end{eqnarray*}
which implies
\begin{eqnarray*}
\partial^2_{xx}\hat b(t, x, \mu, y, s)=\!\!\!\!\!\!\!\!&&\tilde \EE \left[\partial^2_{xx}b(t, x, \mu,Y^{t,x,\mu,y}_s)\right]+\tilde \EE \left[\partial^2_{xy} b(t, x, \mu, Y^{t,x,\mu,y}_s) \cdot \partial_x Y^{t,x,\mu,y}_s\right]\nonumber\\
&&+\tilde \EE \left[\partial^2_{yx} b(t, x, \mu, Y^{t,x,\mu, y}_s)\cdot\partial_{x} Y^{t,x,\mu,y}_{s}\right]\\
&&+\tilde \EE \left[\partial^2_{yy} b(t, x, \mu, Y^{t,x,\mu, y}_s)\cdot(\partial_{x} Y^{t,x,\mu,y}_{s}, \partial_{x} Y^{t,x,\mu,y}_{s})\right]\\
&&+\tilde \EE \left[ \partial_{y} b(t, x, \mu, Y^{t,x,\mu, y}_s)\cdot\partial^2_{xx}Y^{t,x,\mu,y}_{s} \right].
\end{eqnarray*}
Then for any $y_1,y_2\in\RR^m$,
\begin{eqnarray}
&&\|\partial^2_{xx}\hat b(t, x, \mu, y_1, s)-\partial^2_{xx}\hat b(t, x, \mu, y_2, s)\|\nonumber\\
\leq\!\!\!\!\!\!\!\!&&\left\|\tilde \EE \left[\partial^2_{xx}b(t, x, \mu,Y^{t,x,\mu,y_1}_s)-\partial^2_{xx}b(t, x, \mu,Y^{t,x,\mu,y_2}_s)\right]\right\|\nonumber\\
&&+\left\|\tilde \EE \left[\partial^2_{xy} b(t, x, \mu, Y^{t,x,\mu,y_1}_s) \cdot \partial_x Y^{t,x,\mu,y_1}_s-\partial^2_{xy} b(t, x, \mu, Y^{t,x,\mu,y_2}_s) \cdot \partial_x Y^{t,x,\mu,y_2}_s\right]\right\|\nonumber\\
&&+\left\|\tilde \EE \left[\partial^2_{yx} b(t, x, \mu, Y^{t,x,\mu, y_1}_s)\cdot\partial_{x} Y^{t,x,\mu,y_1}_{s}-\partial^2_{yx} b(t, x, \mu, Y^{t,x,\mu, y_2}_s)\cdot\partial_{x} Y^{t,x,\mu,y_2}_{s}\right]\right\|\nonumber\\
&&+\left\|\tilde \EE \left[\partial^2_{yy} b(t, x, \mu, Y^{t,x,\mu, y_1}_s)\cdot(\partial_{x} Y^{t,x,\mu,y_1}_{s}, \partial_{x} Y^{t,x,\mu,y_1}_{s})\right.\right.\nonumber\\
&&\left.\left.\quad\quad\quad-\partial^2_{yy} b(t, x, \mu, Y^{t,x,\mu, y_2}_s)\cdot(\partial_{x} Y^{t,x,\mu,y_2}_{s}, \partial_{x} Y^{t,x,\mu,y_2}_{s})\right]\right\|\nonumber\\
&&+\left\|\tilde \EE \left[ \partial_{y} b(t, x, \mu, Y^{t,x,\mu, y_1}_s)\cdot\partial^2_{xx}Y^{t,x,\mu,y_1}_{s}-\partial_{y} b(t, x, \mu, Y^{t,x,\mu, y_2}_s)\cdot\partial^2_{xx}Y^{t,x,\mu,y_2}_{s} \right]\right\|\nonumber\\
:=\!\!\!\!\!\!\!\!&&\sum^5_{i=1} J_{1i}.\label{J1}
\end{eqnarray}

By condition \eref{A41} and Lemma \ref{L3.6}, there exists $\eta>0$ such that
\begin{eqnarray}
J_{11}\leq C\tilde \EE |Y^{t,x,\mu,y_1}_s-Y^{t,x,\mu,y_2}_s|^{\gamma_2}\leq Ce^{-\eta s}|y_1-y_2|^{\gamma_2}.\label{J11}
\end{eqnarray}

By the boundedness of $\|\partial^2_{xy}b\|$ and condition \eref{A42}, we have
\begin{eqnarray*}
J_{12}\leq\!\!\!\!\!\!\!\!&&\tilde \EE \left\|\partial^2_{xy} b(t, x, \mu, Y^{t,x,\mu,y_1}_s) \cdot \partial_x Y^{t,x,\mu,y_1}_s-\partial^2_{xy} b(t, x, \mu, Y^{t,x,\mu,y_2}_s) \cdot \partial_x Y^{t,x,\mu,y_1}_s\right\|\\
&&+\tilde \EE \left\|\partial^2_{xy} b(t, x, \mu, Y^{t,x,\mu,y_2}_s) \cdot \partial_x Y^{t,x,\mu,y_1}_s-\partial^2_{xy} b(t, x, \mu, Y^{t,x,\mu,y_2}_s) \cdot \partial_x Y^{t,x,\mu,y_2}_s\right\|\\
\leq\!\!\!\!\!\!\!\!&&\tilde \EE \left[\|\partial^2_{xy} b(t, x, \mu, Y^{t,x,\mu,y_1}_s)-\partial^2_{xy} b(t, x, \mu, Y^{t,x,\mu,y_2}_s)\| \|\partial_x Y^{t,x,\mu,y_1}_s\|\right]\\
&&+\tilde \EE \|\partial^2_{xy} b(t, x, \mu, Y^{t,x,\mu,y_2}_s)\| \|\partial_x Y^{t,x,\mu,y_1}_s-\partial_x Y^{t,x,\mu,y_2}_s\|\\
\leq\!\!\!\!\!\!\!\!&&C_T\left[\tilde \EE |Y^{t,x,\mu,y_1}_s-Y^{t,x,\mu,y_2}_s|^{2\gamma_2}\right]^{1/2} \left[\tilde \EE\|\partial_x Y^{t,x,\mu,y_1}_s\|^2\right]^{1/2}\\
&&+C_T\tilde \EE\|\partial_x Y^{t,x,\mu,y_1}_s-\partial_x Y^{t,x,\mu,y_2}_s\|,
\end{eqnarray*}
where $\partial_x Y^{t,x,\mu,y_2}_s$ satisfies
 \begin{equation}\left\{\begin{array}{l}\label{Dx}
\displaystyle
d\partial_{x} Y^{t,x,\mu,y}_s=[\partial_x f(t,x,\mu,Y^{t,x,\mu,y}_{s})+\partial_y f(t,x,\mu,Y^{t,x,\mu,y}_{s})\partial_{x} Y^{t,x,\mu,y}_s]ds\\
\quad\quad\quad\quad\quad\quad+\left[\partial_{x} g(t, x,\mu,Y^{t,x,\mu,y}_s)+\partial_{y} g(t, x,\mu,Y^{t,x,\mu,y}_s)\partial_{x} Y^{t,x,\mu,y}_s\right]d\tilde{W}_{s}^{2},\\
\partial_{x} Y^{t,x,\mu,y}_0=0.\\
\end{array}\right.
\end{equation}
Under the assumptions \ref{A1}, it is easy to prove that
\begin{eqnarray}\label{EDx}
\sup_{t\in[0, T], s\geq 0, x\in\RR^n, y\in\RR^m, \mu\in\mathscr{P}_2}\tilde \EE\|\partial_{x} Y^{t,x,\mu,y}_s\|^4\leq C_T,
\end{eqnarray}
and by Lemma \ref{L3.6} and the boundedness of $\partial_{xy} f$, $\partial_{yy} f$, $\partial_{xy} g$ and $\partial_{yy} g$,  we have
\begin{eqnarray}\label{EDxy}
\sup_{t\in[0, T], x\in\RR^n, \mu\in\mathscr{P}_2}\tilde \EE\|\partial_x Y^{t,x,\mu,y_1}_s-\partial_x Y^{t,x,\mu,y_2}_s\|^2\leq C_T e^{-\frac{\beta s}{2}}|y_1-y_2|^2.
\end{eqnarray}
Then Lemma \ref{L3.6}, \eref{EDx} and \eref{EDxy} imply that there exists $\eta>0$ such that
\begin{eqnarray}
J_{12}\leq  C_T e^{-\eta s}(|y_1-y_2|+1).\label{J12}
\end{eqnarray}

By condition \eref{A42} and a similar arguments as in estimating $J_{12}$, we also have
\begin{eqnarray}
J_{13}\leq Ce^{-\eta s}(|y_1-y_2|+1).\label{J13}
\end{eqnarray}

By condition \eref{A431} and a straightforward computation,
\begin{eqnarray*}
J_{14}\leq\!\!\!\!\!\!\!\!&&\tilde \EE \left\|\partial^2_{yy} b(t, x, \mu, Y^{t,x,\mu,y_1}_s) \cdot (\partial_x Y^{t,x,\mu,y_1}_s, \partial_x Y^{t,x,\mu,y_1}_s)\right.\\
&&\quad\left.-\partial^2_{yy} b(t, x, \mu, Y^{t,x,\mu,y_2}_s) \cdot( \partial_x Y^{t,x,\mu,y_1}_s\, \partial_x Y^{t,x,\mu,y_1}_s)\right\|\\
&&+\tilde \EE \left\|\partial^2_{yy} b(t, x, \mu, Y^{t,x,\mu,y_2}_s) \cdot (\partial_x Y^{t,x,\mu,y_1}_s,  \partial_x Y^{t,x,\mu,y_1}_s)\right.\\
&&\quad\quad\left.-\partial^2_{yy} b(t, x, \mu, Y^{t,x,\mu,y_2}_s) \cdot (\partial_x Y^{t,x,\mu,y_2}_s, \partial_x Y^{t,x,\mu,y_2}_s)\right\|\\
\leq\!\!\!\!\!\!\!\!&&\tilde \EE \left[\|\partial^2_{yy} b(t, x, \mu, Y^{t,x,\mu,y_1}_s)-\partial^2_{yy} b(t, x, \mu, Y^{t,x,\mu,y_2}_s)\| \|\partial_x Y^{t,x,\mu,y_1}_s\|^2\right]\\
&&+\tilde \EE \left[\|\partial^2_{yy} b(t, x, \mu, Y^{t,x,\mu,y_2}_s)\| \|\partial_x Y^{t,x,\mu,y_1}_s-\partial_x Y^{t,x,\mu,y_2}_s\|(\|\partial_x Y^{t,x,\mu,y_1}_s\|+\|\partial_x Y^{t,x,\mu,y_2}_s\|)\right]\\
\leq\!\!\!\!\!\!\!\!&&C_T\left[\tilde \EE |Y^{t,x,\mu,y_1}_s-Y^{t,x,\mu,y_2}_s|^{2\gamma_2}\right]^{1/2} \left[\tilde \EE\|\partial_x Y^{t,x,\mu,y_1}_s\|^4\right]^{1/2}\\
&&+C_T\left[\tilde \EE\|\partial_x Y^{t,x,\mu,y_1}_s-\partial_x Y^{t,x,\mu,y_2}_s\|^2\right]^{1/2}\left[\tilde \EE\left(\|\partial_x Y^{t,x,\mu,y_1}_s\|^2+\|\partial_x Y^{t,x,\mu,y_2}_s\|^2\right)\right]^{1/2}.
\end{eqnarray*}
Then by Lemma \ref{L3.6}, \eref{EDx} and \eref{EDxy}, we get
\begin{eqnarray}
J_{14}\leq  C_T e^{-\eta s}(|y_1-y_2|+1).\label{J14}
\end{eqnarray}

Under the assumptions \ref{A1}-\ref{A3}, it is easy to prove that
\begin{eqnarray}\label{EDxx}
\sup_{t\in[0, T], s\geq 0, x\in\RR^n, y\in\RR^m, \mu\in\mathscr{P}_2}\tilde \EE\|\partial^2_{xx} Y^{t,x,\mu,y}_s\|^2\leq C_T
\end{eqnarray}
and
\begin{eqnarray*}\label{EDxxy}
\sup_{t\in[0, T], x\in\RR^n, \mu\in\mathscr{P}_2}\tilde \EE\|\partial^2_{xx} Y^{t,x,\mu,y_1}_s-\partial^2_{xx} Y^{t,x,\mu,y_2}_s\|^2\leq C_T e^{-2\eta s}|y_1-y_2|^{2\gamma_2}.
\end{eqnarray*}
Then, we get
\begin{eqnarray}
J_{15}\leq C_T e^{-\eta s}|y_1-y_2|^{\gamma_2}.\label{J15}
\end{eqnarray}

Hence, by \eref{J1}, \eref{J11}, \eref{J12}, \eref{J13}, \eref{J14} and \eref{J15} we obtain
\begin{eqnarray*}
J_1\leq Ce^{-\eta s}(\tilde \EE|y-Y^{t,x,\mu,y}_{s_0}|+1)\leq C_Te^{-\eta s}\left\{1+|x|+|y|+[\mu(|\cdot|^2)]^{1/2}\right\}.
\end{eqnarray*}

(ii) For the term $J_2$, note that
\begin{eqnarray*}
\partial^2_{xy}\hat b(t, x, \mu, y, s)=\!\!\!\!\!\!\!\!&& \partial_{y}\tilde \EE \left[\partial_x b(t, x, \mu, Y^{t,x,\mu, y}_{s})\right]+\partial_{y}\tilde\EE\left[ \partial_{y}b(t,x,\mu,Y^{t,x,\mu,y}_{s})\cdot\partial_{x}Y^{t,x,\mu,y}_s \right]\nonumber\\
=\!\!\!\!\!\!\!\!&&\tilde \EE \left[ \partial^2_{xy}b(t, x, \mu, Y^{t,x,\mu,y}_s)\partial_{y}Y^{t,x,\mu,y}_s\right]+\tilde\EE\left[ \partial^2_{yy}b(t,x,\mu,Y^{t,x,\mu,y}_{s})\cdot(\partial_{x}Y^{t,x,\mu,y}_s, \partial_{y}Y^{t,x,\mu,y}_s)\right]\\
&&+\tilde \EE \left[ \partial_{y}b(t, x, \mu, Y^{t,x,\mu,y}_s)\partial^2_{xy}Y^{t,x,\mu,y}_s\right].
\end{eqnarray*}
Lemma \ref{L3.6} and \eref{EDxy} imply
$$
\sup_{t\in[0, T], x\in\RR^n, \mu\in\mathscr{P}_2,y\in\RR^m}\left(\tilde \EE\|\partial_{y}Y^{t,x,\mu,y}_s\|^2+\EE\|\partial^2_{xy}Y^{t,x,\mu,y}_s\|^2\right)\leq C_Te^{-\frac{\beta s}{2}}.
$$
Hence we have
$$
\sup_{t\in[0, T], x\in\RR^n, \mu\in\mathscr{P}_2,y\in\RR^m}\|\partial^2_{xy}\hat b(t, x, \mu, y, s)\|\leq C_Te^{-\frac{\beta s}{4}}.
$$
Hence, it is easy to see that
\begin{eqnarray*}
J_2\leq\!\!\!\!\!\!\!\!&& C_Te^{-\frac{\beta s}{4}}\tilde \EE\|\partial_{x} Y^{t,x,\mu,y}_{s_0}\|\leq C_T e^{\frac{-\beta s}{4}}.
\end{eqnarray*}

(iii) For the term $J_3$, by a similar argument as in (ii), we have
$$
\sup_{t\in[0, T], x\in\RR^n, \mu\in\mathscr{P}_2,y\in\RR^m}\|\partial^2_{yx}\hat b(t, x, \mu, y, s)\|\leq C_Te^{-\frac{\beta s}{4}}
$$
and
$$
\sup_{t\in[0, T], x\in\RR^n, \mu\in\mathscr{P}_2,y\in\RR^m}\|\partial^2_{yy}\hat b(t, x, \mu, y, s)\|\leq C_Te^{-\frac{\beta s}{4}}.
$$
Hence, it is easy to see that
\begin{eqnarray*}
J_3\leq C_T e^{\frac{-\beta s}{4}}.
\end{eqnarray*}

(iv) For the term $J_4$, by estimates \eref{EDxx} and \eref{S1},
we easily get
\begin{eqnarray*}
J_4\leq C_T e^{\frac{-\beta s}{4}}.
\end{eqnarray*}
Hence, combining (i)-(iv), we prove estimate \eref{E21}.

\vspace{0.3cm}
\textbf{Step 3.}  In this step, we intend to prove estimate \eref{E22}. Recall that

\begin{eqnarray*}
\partial_{\mu}\tilde b_{s_0}(t, x, \mu, y, s)(z)=\!\!\!\!\!\!\!\!&& \partial_{\mu} \hat b(t, x, \mu, y, s)(z)-\tilde \EE \partial_{\mu}\hat b(t, x, \mu, Y^{t,x,\mu,y}_{s_0},s)(z)\nonumber\\
&&-\tilde \EE \left[\langle \partial_y\hat b(t, x, \mu, Y^{t,x,\mu, y}_{s_0},s) , \partial_{\mu} Y^{t,x,\mu,y}_{s_0}(z)\rangle \right].
\end{eqnarray*}
So we have
\begin{eqnarray*}
\partial_{z}\partial_{\mu}\tilde b_{s_0}(t, x, \mu, y, s)(z)=\!\!\!\!\!\!\!\!&& \partial_{z}\partial_{\mu} \hat b(t, x, \mu, y, s)(z)-\tilde \EE \partial_{z}\partial_{\mu}\hat b(t, x, \mu, Y^{t,x,\mu,y}_{s_0},s)(z)\nonumber\\
&&-\tilde \EE \left[\langle \partial_y\hat b(t, x, \mu, Y^{t,x,\mu, y}_{s_0},s) , \partial_{z}\partial_{\mu} Y^{t,x,\mu,y}_{s_0}(z)\rangle \right],
\end{eqnarray*}
where $\partial_z\partial_{\mu}Y^{t,x,\mu,y}_s(z)$ satisfies
 \begin{equation}\left\{\begin{array}{l}\label{V}
\displaystyle
d\partial_{z}\partial_{\mu} Y^{t,x,\mu,y}_s(z)=\partial_{z}\partial_{\mu} f(t, x,\mu,Y^{t,x,\mu,y}_{s})(z)ds+\partial_y f(t,x,\mu,Y^{t,x,\mu,y}_{s})\partial_{z}\partial_{\mu} Y^{t,x,\mu,y}_s(z)ds\\
\quad\quad\quad\quad\quad\quad+\left[\partial_{z}\partial_{\mu} g(t, x,\mu,Y^{t,x,\mu,y}_s)(z)+\partial_{y} g(t, x,\mu,Y^{t,x,\mu,y}_s)\partial_{z}\partial_{\mu} Y^{t,x,\mu,y}_s(z)\right]d\tilde{W}_{s}^{2},\\
\partial_{z}\partial_{\mu} Y^{t,x,\mu,y}_s(z)=0.\\
\end{array}\right.
\end{equation}
Under the assumptions \ref{A1}-\ref{A3}, it is easy to prove that for any $T>0$, we have
\begin{eqnarray}
\sup_{t\in [0, T], s\geq 0, x\in\RR^n, y\in\RR^m, \mu\in\mathscr{P}_2}\tilde \EE\|\partial_{z}\partial_{\mu} Y^{t,x,\mu,y}_s\|^2_{L^2(\mu)}\leq C_T \label{Ezm}
\end{eqnarray}
and there exists $\eta>0$ such that
\begin{eqnarray}
\sup_{t\in [0, T], x\in\RR^n,\mu\in\mathscr{P}_2}\tilde\EE\|\partial_{z}\partial_{\mu} Y^{t,x,\mu,y_1}_s-\partial_{z}\partial_{\mu} Y^{t,x,\mu,y_2}_s\|^2_{L^2(\mu)}\leq\!\!\!\!\!\!\!\!&& C_T e^{-2\eta s}|y_1-y_2|^{2\gamma_2}.\label{Ezmy}
\end{eqnarray}
Then we have
\begin{eqnarray*}
&&\|\partial_{z}\partial_{\mu}\hat b(t, x, \mu, y_1, s)-\partial_{z}\partial_{\mu}\hat b(t, x, \mu, y_2, s)\|_{L^2(\mu)}\\
=\!\!\!\!\!\!\!\!&& \|\partial_{z} \partial_{\mu} \tilde\EE b(t, x, \mu, Y^{t,x,\mu, y_1}_s)-\partial_{z}\partial_{\mu}\tilde\EE b(t, x, \mu, Y^{t,x,\mu,y_2}_s)\|_{L^2(\mu)}\nonumber\\
\leq\!\!\!\!\!\!\!\!&& \tilde\EE \left\|\partial_{z}\partial_{\mu} b(t, x, \mu, Y^{t,x,\mu,y_1}_s)-\partial_{z}\partial_{\mu} b(t, x, \mu,Y^{t,x,\mu,y_2}_s)\right\|_{L^2(\mu)}\\
&&+\EE \left\|\partial_y b(t, x, \mu, Y^{t,x,\mu,y_1}_s)\partial_{z}\partial_{\mu} Y^{t,x,\mu,y_1}_s(z)-\partial_y b(t, x, \mu, Y^{t,x,\mu,y_2}_s)\partial_{z}\partial_{\mu} Y^{t,x,\mu,y_2}_s\right\|_{L^2(\mu)}\\
\leq\!\!\!\!\!\!\!\!&& \tilde\EE \left\|\partial_{z}\partial_{\mu} b(t, x, \mu, Y^{t,x,\mu, y_1}_s)- \partial_{z}\partial_{\mu} b(t, x, \mu, Y^{t,x,\mu,y_2}_s)\right\|_{L^2(\mu)}\\
&&+\tilde\EE \left\|\partial_y b(t, x, \mu, Y^{t,x,\mu, y_1}_s)\partial_{z}\partial_{\mu} Y^{t,x,\mu,y_1}_s-\partial_y b(t, x, \mu, Y^{t,x,\mu,y_2}_s)\partial_{z}\partial_{\mu} Y^{t,x,\mu,y_1}_s\right\|_{L^2(\mu)}\\
&&+\tilde\EE \left\|\partial_y b(t, x, \mu, Y^{t,x,\mu,y_2}_s)\partial_{z}\partial_{\mu} Y^{t,x,\mu,y_1}_s-\partial_y b(t, x, \mu, Y^{t,x,\mu,y_2}_s)\partial_{z}\partial_{\mu} Y^{t,x,\mu,y_2}_s\right\|_{L^2(\mu)}\\
:=\!\!\!\!\!\!\!\!&&\sum^{3}_{i=1}K_i.
\end{eqnarray*}

For the terms $K_1$ and $K_2$, it follows from condition \eref{A44} that
\begin{eqnarray}
K_1\leq C_T\tilde\EE|Y^{t,x,\mu,y_1}_s-Y^{t,x,\mu,y_2}_s|^{\gamma_2}\leq C_T e^{-\eta s}|y_1-y_2|^{\gamma_2}\label{K1}
\end{eqnarray}
and by \eref{Ezm}
\begin{eqnarray}
K_2\leq\!\!\!\!\!\!\!\!&&C_T\tilde\EE\|(Y^{t,x,\mu,y_1}_s-Y^{t,x,\mu,y_2}_s)\partial_{z}\partial_{\mu} Y^{t,x,\mu,y_1}_s\|_{L^2(\mu)}\nonumber\\
\leq\!\!\!\!\!\!\!\!&& C_T\left[\tilde\EE|Y^{t,x,\mu,y_1}_s-Y^{t,x,\mu,y_2}_s|^2\right]^{1/2}\left[\tilde \EE\|\partial_{z}\partial_{\mu} Y^{t,x,\mu,y_1}_s\|^2_{L^2(\mu)}\right]^{1/2}\nonumber\\
\leq\!\!\!\!\!\!\!\!&& C_T e^{\frac{-\beta s}{2}}|y_1-y_2|.\label{K2}
\end{eqnarray}

For the term $K_3$, by \eref{Ezmy}, it is easy to see that
\begin{eqnarray}
K_3\leq\!\!\!\!\!\!\!\!&&C_T\tilde\EE\|\partial_{z}\partial_{\mu} Y^{t,x,\mu,y_1}_s-\partial_{z}\partial_{\mu} Y^{t,x,\mu,y_2}_s\|_{L^2(\mu)}\nonumber\\
\leq\!\!\!\!\!\!\!\!&& C_T e^{-\eta s}|y_1-y_2|^{\gamma_2}.\label{K3}
\end{eqnarray}

Therefore, estimates \eref{K1} to \eref{K3} imply
\begin{eqnarray*}
\|\partial_{z}\partial_{\mu}\hat b(t, x, \mu, y_1, s)-\partial_{z}\partial_{\mu}\hat b(t, x, \mu, y_2, s)\|_{L^2(\mu)}\leq C_T e^{-\eta s}(|y_1-y_2|+1).
\end{eqnarray*}
Hence, we finally have
\begin{eqnarray*}
\|\partial_{z}\partial_{\mu}\tilde b_{s_0}(t, x, \mu, y, s)\|_{L^2(\mu)}\leq\!\!\!\!\!\!\!\!&& \|\partial_{z}\partial_{\mu} \hat b(t, x, \mu, y, s)(z)-\tilde \EE \partial_{z}\partial_{\mu}\hat b(t, x, \mu, Y^{t,x,\mu,y}_{s_0},s)(z)\|_{L^2(\mu)}\nonumber\\
&&+\tilde \EE \left[ \|\partial_y\hat b(t, x, \mu, Y^{t,x,\mu, y}_{s_0},s)\| \|\partial_{z}\partial_{\mu} Y^{t,x,\mu,y}_{s_0}\|_{L^2(\mu)}\right]\\
\leq\!\!\!\!\!\!\!\!&& C_T e^{-\eta s}(\tilde \EE |y-Y^{t,x,\mu,y}_{s_0}|+1)+C_T e^{-\eta s}\left[\tilde \EE\|\partial_{z}\partial_{\mu} Y^{t,x,\mu,y}_{s_0}\|^2_{L^2(\mu)}\right]^{1/2}\\
\leq\!\!\!\!\!\!\!\!&& C_T e^{-\eta s}\{1+|x|+|y|+[\mu(|\cdot|^2)]^{1/2}\},
\end{eqnarray*}
which completes the proof of estimate \eref{E22}.
\end{proof}

\vskip 0.2cm
\subsection{The Proof of Theorem \ref{main result 2}}
\begin{proof}
Note that
\begin{eqnarray*}
X_{t}^{\ep}-\bar{X}_{t}=\!\!\!\!\!\!\!\!&&\int_{0}^{t}\left[b(s, X_{s}^{\ep},\mathscr{L}_{X_{s}^{\ep}},Y_{s}^{\ep})-\bar{b}(s,\bar{X}_{s},\mathscr{L}_{\bar X_{s}})\right]ds\\
&&+\int_{0}^{t}\left[\sigma(s, X^{\ep}_{s},\mathscr{L}_{X_{s}^{\ep}})-\sigma(s, \bar{X}_{s},\mathscr{L}_{\bar X_{s}})\right]dW^{1}_s\\
=\!\!\!\!\!\!\!\!&&\int_{0}^{t}\left[b(s, X_{s}^{\ep},\mathscr{L}_{ X_{s}^{\ep}},Y_{s}^{\ep})-\bar{b}(s, X^{\ep}_{s},\mathscr{L}_{X^{\ep}_{s}})\right]ds\\
&&+\int_{0}^{t}\left[\bar{b}(s, X^{\ep}_{s},\mathscr{L}_{X^{\ep}_{s}})-\bar{b}(s, \bar{X}_{s},\mathscr{L}_{\bar{X}_{s}})\right]ds\\
&&+\int_{0}^{t}\left[\sigma(s, X^{\ep}_{s}, \mathscr{L}_{X^{\ep}_{s}})-\sigma(s, \bar{X}_{s},\mathscr{L}_{\bar X_{s}})\right]dW^{1}_s.
\end{eqnarray*}
Then it is easy to see that for any $t\in [0, T]$, we have
\begin{eqnarray*}
\sup_{t\in [0, T]}\EE|X_{t}^{\ep}-\bar{X}_{t}|^2\leq\!\!\!\!\!\!\!\!&&C\sup_{t\in[0,T]}\EE\left|\int_{0}^{t}b(s, X_{s}^{\ep},\mathscr{L}_{ X_{s}^{\ep}},Y_{s}^{\ep})-\bar{b}(s, X^{\ep}_{s},\mathscr{L}_{X^{\ep}_{s}})ds\right|^2\nonumber\\
&&+C_T\EE\int_{0}^{T}|X_{t}^{\ep}-\bar{X}_{t}|^2 dt.
\end{eqnarray*}
Then Grownall's inequality implies that
\begin{eqnarray}
\sup_{t\in [0, T]}\EE|X_{t}^{\ep}-\bar{X}_{t}|^2\leq\!\!\!\!\!\!\!\!&&C_T\sup_{t\in[0,T]}\EE\left|\int_{0}^{t}b(s, X_{s}^{\ep},\mathscr{L}_{ X_{s}^{\ep}},Y_{s}^{\ep})-\bar{b}(s, X^{\ep}_{s},\mathscr{L}_{X^{\ep}_{s}})ds\right|^2.\label{I3.10}
\end{eqnarray}

By Proposition \ref{P3.6}, there exists $\Phi(t,x,\mu,y)$ such that
$$-\mathscr{L}_{2}(t,x,\mu)\Phi(t,x,\mu,y)=b(t,x,\mu,y)-\bar{b}(t,x,\mu).$$
Then by It\^o's formula for a function which depends on measures (see \cite[Theorem 7.1]{BLPR}), we have
\begin{eqnarray*}
\Phi(t,X_{t}^{\ep},\mathscr{L}_{X^{\ep}_{t}},Y^{\ep}_{t})=\!\!\!\!\!\!\!\!&&\Phi(0,x,\delta_{x},y)+\int^t_0 \partial_t \Phi(s, X_{s}^{\ep},\mathscr{L}_{X^{\ep}_{s}},Y^{\ep}_{s})ds\\
&&+\int^t_0 \EE\left[b(s,X^{\ep}_s,\mathscr{L}_{ X^{\ep}_{s}}, Y^{\ep}_s)\partial_{\mu}\Phi(s,x,\mu,y)(X^{\ep}_s)\right]\mid_{x=X_{s}^{\ep},\mu=\mathscr{L}_{X^{\ep}_{s}},y=Y^{\ep}_{s}}ds\\
&&+\int^t_0 \frac{1}{2}\EE \text{Tr}\left[\sigma\sigma^{*}(s,X^{\ep}_s,\mathscr{L}_{ X^{\ep}_{s}})\partial_z\partial_{\mu}\Phi(s,x,\mu,y)(X^{\ep}_s)\right]\mid_{x=X_{s}^{\ep},\mu=\mathscr{L}_{X^{\ep}_{s}},y=Y^{\ep}_{s}}ds\\
&&+\int^t_0 \mathscr{L}_{1}(s,\mathscr{L}_{X^{\ep}_{s}},Y^{\ep}_{s})\Phi(s,X_{s}^{\ep},\mathscr{L}_{X^{\ep}_{s}},Y^{\ep}_{s})ds\\
&&+\frac{1}{\ep}\int^t_0 \mathscr{L}_{2}(s,X_{s}^{\ep},\mathscr{L}_{X^{\ep}_{s}})\Phi(s,X_{s}^{\ep},\mathscr{L}_{X^{\ep}_{s}},Y^{\ep}_{s})ds+M^{\ep,1}_t+\frac{1}{\sqrt{\ep}}M^{\ep,2}_t,
\end{eqnarray*}
where $\mathscr{L}_{1}(t,\mu,y)\Phi(t,x,\mu,y):=(\mathscr{L}_{1}(t,\mu,y)\Phi_1(t,x,\mu,y),\ldots, \mathscr{L}_{2}(t,\mu,y)\Phi_n(t,x,\mu,y))$ with
\begin{eqnarray*}
 \mathscr{L}_{1}(t,\mu,y)\Phi_k(t,x,\mu,y):=\!\!\!\!\!\!\!\!&&\langle b(t,x,\mu,y), \partial_x \Phi(t,x,\mu,y)\rangle\\
 &&+\frac{1}{2}\text{Tr}[\sigma\sigma^{*}(t,x,\mu)\partial^2_{xx} \Phi_k(t,x,\mu,y)],\quad k=1,\ldots, n,
\end{eqnarray*}
and $M^{\ep,1}_t, M^{\ep,2}_t$ are two martingales, which are defined by
\begin{eqnarray*}
&&M^{\ep,1}_t:=\int^t_0 \partial_x \Phi(s,X_{s}^{\ep},\mathscr{L}_{X^{\ep}_{s}})\cdot \sigma(s,X^{\ep}_s,\mathscr{L}_{ X^{\ep}_{s}}) dW^1_s;\\
&&M^{\ep,2}_t:=\int^t_0 \partial_y \Phi(s,X_{s}^{\ep},\mathscr{L}_{X^{\ep}_{s}})\cdot g(s,X^{\ep}_s,\mathscr{L}_{ X^{\ep}_{s}}, Y^{\ep}_s) dW^2_s.
\end{eqnarray*}
Then we have
\begin{eqnarray*}
&&\sup_{t\in[0,T]}\EE\left|\int_{0}^{t}b(s, X_{s}^{\ep},\mathscr{L}_{ X_{s}^{\ep}},Y_{s}^{\ep})-\bar{b}(s, X^{\ep}_{s},\mathscr{L}_{X^{\ep}_{s}})ds\right|^2\\
=\!\!\!\!\!\!\!\!&&\sup_{t\in[0,T]}\EE\left|\int^t_0\mathscr{L}_{2}(s,X_{s}^{\ep},\mathscr{L}_{X^{\ep}_{s}})\Phi(s,X_{s}^{\ep},\mathscr{L}_{X^{\ep}_{s}},Y^{\ep}_{s})ds\right|^2\\
\leq\!\!\!\!\!\!\!\!&&\ep^2\sup_{t\in[0, T]}\EE\left|\Phi(t,X_{t}^{\ep},\mathscr{L}_{X^{\ep}_{t}},Y^{\ep}_{t})-\Phi(0,x,\delta_{x},y)-\int^t_0 \partial_t \Phi(s, X_{s}^{\ep},\mathscr{L}_{X^{\ep}_{s}},Y^{\ep}_{s})ds\right.\\
&&-\int^t_0 \EE\left[b(s,X^{\ep}_s,\mathscr{L}_{ X^{\ep}_{s}}, Y^{\ep}_s)\partial_{\mu}\Phi(s,x,\mu,y)(X^{\ep}_s)\right]\mid_{x=X_{s}^{\ep},\mu=\mathscr{L}_{X^{\ep}_{s}},y=Y^{\ep}_{s}}ds\\
&&-\int^t_0 \EE \text{Tr}\left[\sigma\sigma^{*}(s,X^{\ep}_s,\mathscr{L}_{ X^{\ep}_{s}})\partial_z\partial_{\mu}\Phi(s,x,\mu,y)(X^{\ep}_s)\right]\mid_{x=X_{s}^{\ep},\mu=\mathscr{L}_{X^{\ep}_{s}},y=Y^{\ep}_{s}}ds\\
&&-\left.\int^t_0 \mathscr{L}_{1}(s,\mathscr{L}_{X^{\ep}_{s}},Y^{\ep}_{s})\Phi(s,X_{s}^{\ep},\mathscr{L}_{X^{\ep}_{s}},Y^{\ep}_{s})ds\right|^2\\
&&+\ep^2\sup_{t\in[0, T]}\EE\left|M^{\ep,1}_t\right|^2+\ep\sup_{t\in[0, T]}\EE\left|M^{\ep,2}_t\right|^2.
\end{eqnarray*}
By It\^o's isometry and estimates \eref{E1} and \eref{E2}, we finally get
\begin{eqnarray*}
\sup_{t\in[0,T]}\EE\left|\int_{0}^{t}b(s, X_{s}^{\ep},\mathscr{L}_{ X_{s}^{\ep}},Y_{s}^{\ep})-\bar{b}(s, X^{\ep}_{s},\mathscr{L}_{X^{\ep}_{s}})ds\right|^2
\leq\!\!\!\!\!\!\!\!&&C_T\ep\left[\sup_{t\in[0, T]}\EE|X_{t}^{\ep}|^4+\sup_{t\in[0, T]}\EE|Y^{\ep}_{t}|^4+1\right]\\
\leq\!\!\!\!\!\!\!\!&&C_T(1+|x|^4+|y|^4)\ep.
\end{eqnarray*}
This and \eref{I3.10} imply the assertion.
\end{proof}

\section{Example}

Here we give a simple example as an application of our results.

\begin{example} Let $b_0:\RR^n\times\RR^m\rightarrow \RR^n$, $f_0:\RR^n\times\RR^m\rightarrow \RR^m$ and satisfying the following conditions:

\vspace{0.2cm}
(1) The first-order partial derivatives $\partial_{x}b_0(x,y), \partial_{y}b_0(x,y), \partial_{x}f_0(x,y), \partial_{y}f_0(x,y)$ exist for any $x\in \RR^n, y\in\RR^m$. Moreover, all these first-order partial derivatives are bounded uniformly in $(x,y)$ and Lipschitz continuous $w.r.t.$ $y$ uniformly in $x$.

\vspace{0.2cm}
(2) There exists $\beta>0$ such that for any $x\in\RR^n$ and $y_1,y_2\in\RR^m$,
\begin{eqnarray*}
\langle f_0(x, y_1)-f_0(x,y_2), y_1-y_2\rangle\leq -\beta |y_1-y_2|^2;
\end{eqnarray*}

\vspace{0.2cm}
(3) The second-order partial derivatives $\partial^2_{xx} b_0(x,y)$, $\partial^2_{xy} b_0(x,y)$, $\partial^2_{xx} f_0(x,y)$ and $\partial^2_{xy} f_0(x,y)$ exist for any $x\in \RR^n, y\in\RR^m$. Moreover, all these second-order partial derivatives are bounded uniformly in $(x,y)$ and Lipschitz continuous $w.r.t.$ $y$ uniformly in $x$.

Now, let us consider the following slow-fast distribution dependent stochastic differential equations,
\begin{equation}\left\{\begin{array}{l}\label{EX1}
\displaystyle
d X^{\ep}_t = b( X^{\ep}_t,\mathscr{L}_{X^{\ep}_t},Y^{\ep}_t) dt+ d W^{1}_t,\quad X^{\ep}_0=x\in \RR^n\\
\displaystyle
d Y^{\ep}_t =\frac{1}{\ep} f( X^{\ep}_t,\mathscr{L}_{X^{\ep}_t},Y^{\ep}_t) dt+\frac{1}{\sqrt{\ep}}d W^{2}_t,\quad Y^{\ep}_0=y\in \RR^m,
\end{array}\right.
\end{equation}
where $\{W^{1}_t\}_{t\geq 0}$ and $\{W^{2}_t\}_{t\geq 0}$ are mutually independent $n-$ and $m-$ dimensional standard Brownian motions and
$$
b(x,\mu,y):=\int_{\RR^n}b_0(x+z,y)\mu(dz),\quad
f(x,\mu,y):=\int_{\RR^n}f_0(x+z,y)\mu(dz).
$$
Then we have
$$
\partial_{\mu}b(x,\mu,y)(\cdot)=\partial_x b_0(x+\cdot,y),\quad \partial_{z}\partial_{\mu}b(x,\mu,y)(z)=\partial^2_{xx} b_0(x+z,y)
$$
and
$$
\partial_{\mu}f(x,\mu,y)(\cdot)=\partial_x f_0(x+\cdot,y)\quad \partial_{z}\partial_{\mu}f(x,\mu,y)(z)=\partial^2_{xx} f_0(x+z,y).
$$

If the conditions (1) and (2) hold, it is easy to check that the coefficients above satisfy assumptions \ref{A1}-\ref{A2}. Hence, by Theorem \ref{main result 1}, we have
\begin{eqnarray*}
\sup_{t\in[0,T]}\mathbb{E}|X_{t}^{\ep}-\bar{X}_{t}|^2\leq C\ep^{2/3},
\end{eqnarray*}
where $\bar{X}$ solves the corresponding averaged equation.

If the conditions (1)-(3) hold, it is easy to check that the coefficients above satisfy assumptions \ref{A1}-\ref{A3}. Hence, by Theorem \ref{main result 2}, we have
\begin{eqnarray*}
\sup_{t\in[0,T]}\mathbb{E}|X_{t}^{\ep}-\bar{X}_{t}|^2\leq C\ep,
\end{eqnarray*}
where $\bar{X}$ solves the corresponding averaged equation.

\end{example}

\section{Appendix}

In this section, by using the result due to Wang  in  \cite{WFY}, we prove the existence and uniqueness of solutions to system (\ref{Equation}) and the corresponding averaged equation.

\subsection{Proof of Theorem \ref{main}}
\begin{proof}
We set
$$
Z^{\ep}_t:=\left(
                               \begin{array}{c}
                                 X^{\ep}_t \\
                                 Y^{\ep}_t \\
                               \end{array}
                             \right)
,\quad \tilde{b}^{\ep}(t,x,y,\tilde \mu):=\left(
                             \begin{array}{c}
                               b(t,x,\mu,y) \\
                               \frac{1}{\ep}f(t,x,\mu,y) \\
                             \end{array}
                           \right)
$$
and
$$
\tilde{\sigma}^{\ep}(t,x,y,\tilde \mu):=\left(
                                  \begin{array}{cc}
                                    \sigma(t,x,\mu) & 0 \\
                                    0 & \frac{1}{\sqrt{\ep}}g(t,x,\mu,y) \\
                                  \end{array}
                                \right)
, \quad W_t:=\left(
             \begin{array}{c}
               W^1_t \\
               W^2_t \\
             \end{array}
           \right).
$$
where $t\geq 0$, $x\in\RR^n$, $y\in\RR^m$, $\tilde \mu\in\mathscr{P}_{2}(\RR^{n+m})$ with its marginal distribution $\mu$ on $\RR^n$.
Then system \eref{Equation} can be rewritten as the following equation:
\begin{equation}
 dZ^{\ep}_t=\tilde{b}^{\ep}(t, Z^{\ep}_t, \mathscr{L}_{Z^{\ep}_t})dt+\tilde{\sigma}^{\ep}(t, Z^{\ep}_t,\mathscr{L}_{Z^{\ep}_t})d W_t,\quad
Z^{\ep}_{0}=\left(
                      \begin{array}{c}
                        x \\
                        y \\
                      \end{array}
                    \right)
.\label{Eq2}
\end{equation}
Under the assumption \ref{A1}, we intend to prove that the coefficients in equation \eref{Eq2} satisfy Lipschitz and linear growth conditions, uniformly $w.r.t.$ $t\in[0, T]$.

In fact, for  $T>0$, and any $z_i=(x_i, y_i)\in \RR^{n+m}$, $\tilde \mu_i\in\mathscr{P}_{2}(\RR^{n+m})$ with its marginal distributions $\mu_i$ on $\RR^n$, $i=1,2$, $t\in [0,T]$
\begin{eqnarray*}
&&|\tilde{b}^{\ep}(t,z_1,\tilde \mu_1)-\tilde{b}^{\ep}(t,z_2,\tilde \mu_2)|+\|\tilde{\sigma}^{\ep}(t,z_1,\tilde \mu_1)-\tilde{\sigma}^{\ep}(t,z_2,\tilde \mu_2)\|\\
\leq\!\!\!\!\!\!\!\!&&|b(t,x_1,\mu_1, y_1)-b(t,x_2, \mu_2, y_2)|+\|\sigma(t,x_1,\mu_1)-\sigma(t,x_2,\mu_2)\|\\
&&+\frac{1}{\ep}|f(t,x_1, \mu_1,y_1)-f(t,x_2, \mu_2,y_2)|+\frac{1}{\ep}\|g(t,x_1,\mu_1, y_1)-g(t,x_2, \mu_2,y_2)\|\\
\leq\!\!\!\!\!\!\!\!&& C_T\left(1+\frac{1}{\ep}\right)\big[|x_1-x_2|+|y_1-y_2|+\mathbb{W}_2(\mu_1,\mu_2)\big]\\
\leq\!\!\!\!\!\!\!\!&& C_T\left(1+\frac{1}{\ep}\right)\big[|z_1-z_2|+\mathbb{W}_2(\tilde \mu_1,\tilde \mu_2)\big].
\end{eqnarray*}
Furthermore,
\begin{eqnarray*}
&&|\tilde{b}^{\ep}(t, z_1,\tilde \mu_1)|+\|\tilde{\sigma}^{\ep}(t, z_1,\tilde \mu_1)\|\\
\leq\!\!\!\!\!\!\!\!&&|b(t, x_1, \mu_1, y_1)|+\|\sigma(t, x_1,\mu_1)\|+\frac{1}{\ep}|f(t,x_1,\mu_1, y_1)|+\frac{1}{\ep}\|g(t,x_1,\mu_1, y_1)\|\\
\leq\!\!\!\!\!\!\!\!&&C_T\left(1+\frac{1}{\ep}\right)\big[1+|x_1|+|y_1|+\mu_1(|\cdot|^2)\big]\\
\leq\!\!\!\!\!\!\!\!&&C_T\left(1+\frac{1}{\ep}\right)\big[1+|z_1|+\tilde \mu_1(|\cdot|^2)\big].
\end{eqnarray*}
Hence by \cite[Theorem 4.1]{WFY}, there exists a unique solution $\{(X^{\ep}_t,Y^{\ep}_t), t\geq 0\}$ to system (\ref{Equation}).
The proof is complete.
\end{proof}

\subsection{Proof of Lemma \ref{PMA}}

\begin{proof}
We first check that the coefficients of Eq. (\ref{3.1}) satisfy the following condition:

For any $T>0$, there exists $C_T>0$ such that for any $t_i\in [0, T]$, $x_i\in\RR^n$, $\mu_i\in\mathscr{P}_2$, $i=1,2$,
\begin{eqnarray}
&&|\bar{b}(t_1,x_1,\mu_1)-\bar{b}(t_2,x_2,\mu_2)|+\|\sigma(t_1,x_1,\mu_1)-\sigma(t_2,x_2,\mu_2)\|\nonumber\\
\leq\!\!\!\!\!\!\!\!&& C_T\left[|t_1-t_2|+|x_1-x_2|+\mathbb{W}_2(\mu_1,\mu_2)\right].\label{4.16}
\end{eqnarray}

Indeed, by Proposition \ref{Rem 4.1} and Lemma \ref{L3.6}, for any $s>0$, we have
\begin{eqnarray*}
&&|\bar{b}(t_1,x_1,\mu_1)-\bar{b}(t_2,x_2,\mu_2)|+\|\sigma(t_1,x_1,\mu_1)-\sigma(t_2,x_2,\mu_2)\|\\
\leq\!\!\!\!\!\!\!\!&&\left|\bar b(t_1,x_1,\mu_1)-\tilde\EE b(t_1,x_1,\mu_1, Y^{t_1,x_1,\mu_1,0}_s)\right|+\left|\tilde \EE b(t_2,x_2, \mu_2,Y^{t_2,x_2,\mu_2,0}_s)-\bar b(t_2,x_2,\mu_2,)\right|\\
&&+\tilde \EE \left|b(t_1,x_1,\mu_1, Y^{t_1,x_1,\mu_1,0}_s)-b(t_2,x_2, \mu_2,Y^{t_2,x_2,\mu_2,0}_s)\right|+\|\sigma(t_1,x_1,\mu_1)-\sigma(t_2,x_2,\mu_2)\|\\
\leq\!\!\!\!\!\!\!\!&&C_Te^{-\frac{\beta s}{2}}\left(1+|x_1|+|x_2|+[\mu_1(|\cdot|^2)]^{1/2}+[\mu_2(|\cdot|^2)]^{1/2}\right)\\
&&+C_T\left\{|t_1-t_2|+|x_1-x_2|+\tilde \EE|Y^{t_1,x_1,\mu_1,0}_s-Y^{t_2,x_2,\mu_2,0}_s|+\mathbb{W}_2(\mu_1,\mu_2)\right\}\\
\leq\!\!\!\!\!\!\!\!&&C_Te^{-\frac{\beta s}{2}}\left\{1+|x_1|+|x_2|+[\mu_1(|\cdot|^2)]^{1/2}+[\mu_2(|\cdot|^2)]^{1/2}\right\}\\
&&+C_T\left[|t_1-t_2|+|x_1-x_2|+\mathbb{W}_2(\mu_1,\mu_2)\right].
\end{eqnarray*}
Then  \eref{4.16} follows by letting $s\rightarrow \infty$. Moreover, the estimate \eref{4.16} implies
\begin{eqnarray}
|\bar{b}(t_1,x_1,\mu_1)|+\|\sigma(t_1, x_1,\mu_1)\|\leq C_T\left\{1+|x_1|+[\mu_1(|\cdot|^2)]^{1/2}\right\}.\label{4.17}
\end{eqnarray}

Hence by \cite[Theorem 4.1]{WFY}, there exists a unique solution $\{\bar X_t, t\geq 0\}$ to Eq. (\ref{3.1}) and (\ref{3.9}) can be easily obtained by following the same arguments as in the proof of Lemma \ref{PMY}. The proof is complete.
\end{proof}

\subsection{Proof of \eref{tilde bt}-\eref{tilde bmu}}\label{sub6.3}
\begin{proof}
We here only prove \eref{tilde bmu}. \eref{tilde bt} and \eref{tilde bx} can be proved by the same procedure.
For any $s_0>0$, we define
\begin{eqnarray*}
\tilde b_{s_0}(t, x, \mu, y, s):=\hat b(t, x, \mu, y, s)-\hat b(t, x, \mu, y, s+s_0),
\end{eqnarray*}
where $\hat b(t, x, \mu, y, s):=\tilde \EE b(t,x, \mu, Y^{t,x,\mu, y}_s)$. The Proposition \ref{Rem 4.1} implies that
$$
\lim_{s_0\rightarrow \infty} \tilde b_{s_0}(t, x, \mu, y, s)=\tilde b(t,x, \mu, y,s).
$$
As a result, in order to prove \eref{tilde bmu}, it suffices to show there exists $\eta>0$ such that for any $s_0>0$, $t\in[0,T], s\geq 0$, $x\in\RR^n$, $y\in\RR^m$ and $\mu_1,\mu_2\in\mathscr{P}_2$,
\begin{eqnarray*}
&&|\tilde b_{s_0}(t, x, \mu_1, y, s)-\tilde b_{s_0}(t, x, \mu_2,y, s)|\nonumber\\
\leq\!\!\!\!\!\!\!\!&&C_T\mathbb{W}_2(\mu_1,\mu_2)e^{-\eta s}\left\{1+|x|^{\gamma_1}+|y|^{\gamma_1}+[\mu_1(|\cdot|^2)]^{\gamma_1/2}+[\mu_2(|\cdot|^2)]^{\gamma_1/2}\right\},
\end{eqnarray*}
which can be obtained by
\begin{eqnarray}
\sup_{t\in [0, T]}\|\partial_{\mu}\tilde b_{s_0}(t, x, \mu, y, s)\|_{L^2(\mu)}\leq C_Te^{-\eta s}\left\{1+|x|^{\gamma_1}+|y|^{\gamma_1}+[\mu(|\cdot|^2)]^{\gamma_1/2}\right\}. \label{tilde b0}
\end{eqnarray}

Indeed, by the Markov property,
\begin{eqnarray}
\tilde b_{s_0}(t, x, \mu, y, s)=\!\!\!\!\!\!\!\!&& \hat b(t, x, \mu, y, s)-\tilde \EE b(t, x, \mu, Y^{t,x,\mu,y}_{s+s_0})\nonumber\\
=\!\!\!\!\!\!\!\!&& \hat b(t, x, \mu, y, s)-\tilde \EE \{\tilde \EE[b(t, x, \mu, Y^{t,x,\mu,y}_{s+s_0})|\tilde{\mathscr{F}}_{s_0}]\}\nonumber\\
=\!\!\!\!\!\!\!\!&& \hat b(t, x, \mu, y, s)-\tilde \EE \hat b(t, x, \mu, Y^{t,x,\mu,y}_{s_0},s).\label{MP}
\end{eqnarray}
Then we obtain
\begin{eqnarray}
\partial_{\mu}\tilde b_{s_0}(t, x, \mu, y, s)=\!\!\!\!\!\!\!\!&& \partial_{\mu} \hat b(t, x, \mu, y, s)-\tilde \EE \partial_{\mu}\hat b(t, x, \mu, Y^{t,x,\mu,y}_{s_0},s)\nonumber\\
&&-\tilde \EE \left[\langle \partial_y\hat b(t, x, \mu, Y^{t,x,\mu, y}_{s_0},s) , \partial_{\mu} Y^{t,x,\mu,y}_{s_0}\rangle \right].\label{5.8}
\end{eqnarray}
Next, we intend to prove the following two statements.
\begin{itemize}
\item{For any $t\in[0, T], s\geq 0$, $x\in\RR^n$, $y\in\RR^m$ and $\mu\in\mathscr{P}_2$,
\begin{eqnarray}
\|\partial_y\hat b(t, x, \mu, y, s)\|\leq C_T e^{-\frac{\beta s}{2}}.\label{S1}
\end{eqnarray}}
\item{For any $t\in[0,T], s\geq 0$, $x\in\RR^n$, $y_1,y_2\in\RR^m$ and $\mu\in\mathscr{P}_2$,
\begin{eqnarray}
\|\partial_{\mu}\hat b(t, x, \mu, y_1, s)-\partial_{\mu}\hat b(t, x, \mu, y_2, s)\|_{L^2(\mu)}\leq C_T e^{-\eta s}|y_1-y_2|.\label{S2}
\end{eqnarray}
}
\end{itemize}

For the first statement, by Lemma \ref{L3.6},
\begin{eqnarray*}
|\hat b(t, x, \mu, y_1, s)-\hat b(t, x, \mu, y_2, s)|=\!\!\!\!\!\!\!\!&& |\tilde \EE b(t, x, \mu, Y^{t,x,\mu,y_1}_s)-\tilde \EE  b(t, x, \mu,Y^{t,x,\mu,y_2}_s)|\nonumber\\
\leq\!\!\!\!\!\!\!\!&& C_T\tilde\EE|Y^{t,x,\mu,y_1}_s-Y^{t,x,\mu,y_2}_s|\nonumber\\
\leq\!\!\!\!\!\!\!\!&&C_T e^{-\frac{\beta s}{2}}|y_1-y_2|,
\end{eqnarray*}
which implies \eref{S1}.

For the second statement, the assumptions \ref{A1} and \ref{A2} imply $Y^{t,x,\mu,y}_s$ that is differentiable $w.r.t$ $\mu$ and its derivative $\partial_{\mu} Y^{t,x,\mu,y}_s(z)$  satisfies
 \begin{equation}\left\{\begin{array}{l}\label{V}
\displaystyle
d\partial_{\mu} Y^{t,x,\mu,y}_s(z)=\partial_{\mu} f(t, x,\mu,Y^{t,x,\mu,y}_{s})(z)ds+\partial_y f(t,x,\mu,Y^{t,x,\mu,y}_{s})\partial_{\mu} Y^{t,x,\mu,y}_s(z)ds\\
\quad\quad\quad\quad\quad\quad+\left[\partial_{\mu} g(t, x,\mu,Y^{t,x,\mu,y}_s)(z)+\partial_{y} g(t, x,\mu,Y^{t,x,\mu,y}_s)\partial_{\mu} Y^{t,x,\mu,y}_s(z)\right]d\tilde{W}_{s}^{2},\\
\partial_{\mu} Y^{t,x,\mu,y}_s(z)=0.\\
\end{array}\right.
\end{equation}
Moreover, it is easy to see that for any $T>0$,  there exists $C_T$ such that
$$
\sup_{t\in[0, T], s\geq 0, x\in\RR^n, y\in\RR^m, \mu\in\mathscr{P}_2}\tilde \EE\|\partial_{\mu} Y^{t,x,\mu,y}_s\|^2_{L^2(\mu)}\leq C_T.
$$
Then we have
\begin{eqnarray*}
&&\|\partial_{\mu}\hat b(t, x, \mu, y_1, s)-\partial_{\mu}\hat b(t, x, \mu, y_2, s)\|_{L^2(\mu)}\\
=\!\!\!\!\!\!\!\!&& \| \partial_{\mu} \tilde\EE b(t, x, \mu, Y^{t,x,\mu, y_1}_s)-\partial_{\mu}\tilde\EE b(t, x, \mu, Y^{t,x,\mu,y_2}_s)\|_{L^2(\mu)}\nonumber\\
\leq\!\!\!\!\!\!\!\!&& \tilde\EE \left\|\partial_{\mu} b(t, x, \mu, Y^{t,x,\mu,y_1}_s)-\partial_{\mu} b(t, x, \mu,Y^{t,x,\mu,y_2}_s)\right\|_{L^2(\mu)}\\
&&+\EE \left\|\partial_y b(t, x, \mu, Y^{t,x,\mu,y_1}_s)\partial_{\mu} Y^{t,x,\mu,y_1}_s-\partial_y b(t, x, \mu, Y^{t,x,\mu,y_2}_s)\partial_{\mu} Y^{t,x,\mu,y_2}_s\right\|_{L^2(\mu)}\\
\leq\!\!\!\!\!\!\!\!&& \tilde\EE \left\|\partial_{\mu} b(t, x, \mu, Y^{t,x,\mu, y_1}_s)- \partial_{\mu} b(t, x, \mu, Y^{t,x,\mu,y_2}_s)\right\|_{L^2(\mu)}\\
&&+\tilde\EE \left\|\partial_y b(t, x, \mu, Y^{t,x,\mu, y_1}_s)\partial_{\mu} Y^{t,x,\mu,y_1}_s-\partial_y b(t, x, \mu, Y^{t,x,\mu,y_2}_s)\partial_{\mu} Y^{t,x,\mu,y_1}_s\right\|_{L^2(\mu)}\\
&&+\tilde\EE \left\|\partial_y b(t, x, \mu, Y^{t,x,\mu,y_2}_s)\partial_{\mu} Y^{t,x,\mu,y_1}_s-\partial_y b(t, x, \mu, Y^{t,x,\mu,y_2}_s)\partial_{\mu} Y^{t,x,\mu,y_2}_s\right\|_{L^2(\mu)}\\
:=\!\!\!\!\!\!\!\!&&\sum^{3}_{i=1}S_i.
\end{eqnarray*}
For the terms $S_1$ and $S_2$, it follows from conditions \eref{A33}, \eref{A34} and Lemma \ref{L3.6} that there exists $\eta>0$ such that
\begin{eqnarray}
S_1\leq C_T\tilde\EE|Y^{t,x,\mu,y_1}_s-Y^{t,x,\mu,y_2}_s|^{\gamma_1}\leq C_T e^{-\eta s}|y_1-y_2|^{\gamma_1}\label{S_1}
\end{eqnarray}
and
\begin{eqnarray}
S_2\leq\!\!\!\!\!\!\!\!&& C_T\left[\tilde\EE|Y^{t,x,\mu,y_1}_s-Y^{t,x,\mu,y_2}_s|^{2\gamma_1}\right]^{1/2}\left[\tilde \EE\|\partial_{\mu} Y^{t,x,\mu,y_1}_s\|^2_{L^2(\mu)}\right]^{1/2}\nonumber\\
\leq\!\!\!\!\!\!\!\!&& C_T e^{-\eta s}|y_1-y_2|^{\gamma_1}.\label{S_2}
\end{eqnarray}
For the term $S_3$, by a straightforward computer, we obtain that
\begin{eqnarray*}
\tilde\EE\|\partial_{\mu} Y^{t,x,\mu,y_1}_s-\partial_{\mu} Y^{t,x,\mu,y_2}_s\|^2_{L^2(\mu)}\leq\!\!\!\!\!\!\!\!&& C_T e^{-\frac{\beta s}{2}}|y_1-y_2|^{2\gamma_1},
\end{eqnarray*}
which implies
\begin{eqnarray}
S_3\leq C_T\tilde\EE\|\partial_{\mu} Y^{t,x,\mu,y_1}_s-\partial_{\mu} Y^{t,x,\mu,y_2}_s\|_{L^2(\mu)}
\leq C_T e^{\frac{-\beta s}{4}}|y_1-y_2|^{\gamma_1}.\label{S_3}
\end{eqnarray}
Therefore, estimates \eref{S_1} to \eref{S_3} imply \eref{S2}.

Finally, by estimates  \eref{5.8}, \eref{S1} and \eref{S2},  there exists $\eta>0$  such that
\begin{eqnarray*}
\|\partial_{\mu}\tilde b_{s_0}(t, x, \mu, y, s)\|_{L^2(\mu)}\leq\!\!\!\!\!\!\!\!&& Ce^{-\eta s}\tilde \EE|y-Y^{t,x,\mu,y}_{s_0}|^{\gamma_1}+Ce^{-\eta t}\\\leq\!\!\!\!\!\!\!\!&&C_Te^{-\eta s}\left\{1+|x|^{\gamma_1}+|y|^{\gamma_1}+[\mu(|\cdot|^2)]^{\gamma_1/2}\right\},
\end{eqnarray*}
which proves \eref{tilde b0}. The proof is complete.

\end{proof}
\vskip 0.2cm

\vskip 0.5cm
\textbf{Acknowledge}.
X. Sun is supported by the NNSF of China (No. 11601196) and NSF of Jiangsu Province (No. BK20160004); Y. Xie is supported by the NNSF of China (No. 11771187,11931004). The PAPD of Jiangsu Higher Education Institutions and financial support by the DFG through CRC 1283 are also gratefully acknowledged.

\end{document}